\documentclass[a4paper]{amsart}

\usepackage{color}
\usepackage{amssymb,amsmath,amsthm,amstext,amsfonts,amscd}
\usepackage{psfrag}
\usepackage{url}
\usepackage{amsfonts}
\usepackage{mathrsfs}
\usepackage{graphicx}
\usepackage{enumerate}
\usepackage{fancyhdr}
\usepackage[colorlinks,linkcolor=black,anchorcolor=black,citecolor=black,hyperindex=true,CJKbookmarks=true]{hyperref}

\usepackage{epsfig}
\usepackage[T1]{fontenc}

\makeatletter \@addtoreset{equation}{section} \makeatother

\renewcommand\thetable{\thesection.\@arabic\c@table}

\usepackage{tikz}
\usetikzlibrary{shapes, arrows, positioning, shadows}

\makeatletter
\@namedef{subjclassname@2020}{2020 Mathematics Subject Classification}
\makeatother

\newtheorem{lemma}{Lemma}[section]
\newtheorem{theorem}[lemma]{Theorem}
\newtheorem{proposition}[lemma]{Proposition}
\newtheorem{corollary}[lemma]{Corollary}
\newtheorem{question}{Question}

\newtheorem{maintheorem}{Theorem}

\newtheorem{maincorollary}{Corollary}

\theoremstyle{remark}
\newtheorem{definition}[lemma]{Definition}
\newtheorem{remark}[lemma]{Remark}

\makeatletter
\def\lst@lettertrue{\let\lst@ifletter\iffalse}
\makeatother

\numberwithin{figure}{section}
\title{Distributional Chaos in Over-Saturated Sets}

\author{Chenwei Yu}

\address{Chenwei Yu, School of Mathematical Sciences,  Fudan University\\Shanghai 200433, People's Republic of China}
\email{23110180052@m.fudan.edu.cn}

\begin{document}

\renewcommand{\thepage}{\arabic{page}}
\setcounter{page}{1}


\maketitle

\begin{abstract}
	We study distributional chaos in over-saturated sets under the non-uniform specification property.
	We prove that for any dynamical system with the non-uniform specification property, every non-empty over-saturated set exhibits distributional chaos of type 1.
	We further provide sufficient conditions for certain differences of over-saturated sets to exhibit distributional chaos of type 1.
	We also construct examples showing that some of these conditions are optimal.
\end{abstract}

\thispagestyle{empty}

\section{Introduction}

Throughout this paper, we assume that $(X,f)$ is a dynamical system,
where $X$ is a non-degenerate (i.e. with at least two points) compact space equipped with metric $d$ and $f:X\rightarrow X$ is a continuous map.
Denote by $\mathfrak{B}(X)$ the Borel $\sigma$-algebra of $X$.
Let $\mathcal{M}(X)$, $\mathcal{M}_f(X)$ and $\mathcal{M}_f^{erg}(X)$ denote the spaces of Borel probability measures, $f$-invariant Borel probability measures and $f$-ergodic Borel probability measures on $X$, respectively, all equipped with weak* topology.
Let $\mathbb{Z}$, $\mathbb{N}$ and $\mathbb{N}_0$ denote the sets of integers, positive integers and non-negative integers, respectively.
For any set $S$, let $|S|$ denote the cardinality of $S$.

Let $C(X)$ be the Banach algebra of real-valued continuous functions on $X$ with the supremum norm $\|\cdot\|$, i.e., for $\varphi\in C(X)$, $\|\varphi\|:=\sup_{x\in X}|\varphi(x)|$.
For any $\varphi\in C(X)$, define the {\it $\varphi$-irregular set} $I_\varphi(f)$ by
\begin{equation*}
	I_{\varphi}(f)=\bigg\{x\in X:\lim_{n\rightarrow\infty}\frac1n\sum_{k=0}^{n-1}\varphi(f^kx)\;\text{does not exist}\bigg\}.
\end{equation*}
The \emph{irregular set} $\mathrm{IR}(f)$ and \emph{completely-irregular set} $\mathrm{CI}(f)$ of $(X,f)$ are defined by
\begin{equation*}
	\mathrm{IR}(f)=\bigcup_{\varphi\in C(X)} I_{\varphi}(f)\quad\text{and}\quad\mathrm{CI}(f)=\bigcap_{\varphi\in\mathfrak{C}(f)}I_{\varphi}(f),
\end{equation*}
respectively, where $\mathfrak{C}(f)=\{\varphi\in C(X):I_{\varphi}(f)\neq\varnothing\}$.

Although Birkhoff's ergodic theorem implies that $\mu(\mathrm{IR}(f))=0$ for every $\mu\in\mathcal{M}_f(X)$, various aspects of the dynamical complexity of $\varphi$-irregular sets have been extensively studied.
The specification property, introduced by Bowen \cite{Bowen-1971}, and its variants have played an important role in these studies.
For dynamical systems with the specification property, every non-empty $\varphi$-irregular set is residual in $X$ \cite{Li-Wu-2014}, carries full topological entropy \cite{Chen-Kupper-Shu-2005} and full topological pressure \cite{Thompson-2010}.
Under the almost specification property, every $\varphi$-irregular set is either empty or carries full topological entropy \cite{Thompson-2012}.

In this paper, we consider dynamical complexity in the sense of distributional chaos.
The notion of chaos was first introduced in mathematical language by Li and Yorke \cite{Li-Yorke-1975}.
A set $S$ is called \emph{scrambled} if for any $x,y\in S$ with $x\neq y$,
\begin{equation*}
	\liminf_{n\rightarrow\infty}d(f^nx,f^ny)=0
	\quad\text{and}\quad
	\limsup_{n\rightarrow\infty}d(f^nx,f^ny)>0.
\end{equation*}
A dynamical system is called \emph{Li-Yorke chaotic} if there exists an uncountable scrambled subset $S\subset X$.
In \cite{Schweizer-Smital-1994}, Schweizer and Sm\'ital introduced the notion of distributional chaos, which is one of the most extensively studied concepts of chaos in the sense of Li and Yorke.
There are three types of distributional chaos: DC1 (distributional chaos of type 1), DC2 and DC3.
We focus on distributional chaos of type 1 (DC1), which is the strongest form of distributional chaos.
A pair of points $\{x,y\}$ is called \emph{DC1-scrambled} if the following holds:
\begin{itemize}
	\item for any $t>0$
	\begin{equation*}
		\limsup_{n\rightarrow\infty}\frac1n|\{j\in[0,n-1]:d(f^jx,f^jy)<t\}|=1;
	\end{equation*}
	\item there exists $t_0>0$ such that
	\begin{equation*}
		\liminf_{n\rightarrow\infty}\frac1n|\{j\in[0,n-1]:d(f^jx,f^jy)<t_0\}|=0.
	\end{equation*}
\end{itemize}

\begin{definition}
	A set $S$ is called a \emph{DC1-scrambled set} if each pair of distinct points in $S$ is DC1-scrambled.
\end{definition}

Chen and Tian \cite{Chen-Tian-2021} studied irregular sets and level sets from the viewpoint of DC1 for dynamical systems with the specification property.
They proved the existence of uncountable DC1-scrambled subsets in both kinds of sets.
Subsequently, Hou, Tian and Zhao \cite{Hou-Tian-Zhao-2025} studied the chaoticity of generic points for ergodic measures and obtained almost DC1 and Banach DC1 under suitable hyperbolic assumptions.
More recently, Hou, Tian and Zhao \cite{Hou-Tian-Zhao-2026} considered DC1 in shrinking target problems.
Under certain forms of the exponential specification property, they showed that the corresponding shrinking target sets are DC1.

Let $\delta_x$ denote the Dirac measure at $x$.
For $x\in X$ and $n\in\mathbb{N}$, let
\begin{equation*}
	\delta_x^n=\frac 1n\sum_{k=0}^{n-1}\delta_{f^kx}.
\end{equation*}
Let $V_f(x)$ be the set of accumulation points of $\{\delta_x^n:n\geq 1\}$.
It follows from \cite[Proposition~3.8]{DGS1976} that every $V_f(x)$ is a non-empty compact connected subset of $\mathcal{M}_f(X)$.
For any non-empty connected compact subset $K\subset\mathcal{M}_f(X)$, define the \emph{saturated set} $G_K$ by
\begin{equation*}
	G_K=\{x\in X:V_f(x)=K\}.
\end{equation*}
For any non-empty subset $K\subset\mathcal{M}_f(X)$, define the \emph{over-saturated set} $G^K$ by
\begin{equation*}
	G^K=\{x\in X:V_f(x)\supset K\}.
\end{equation*}
Write $G_{\mu}=G_{\{\mu\}}$ and $G^{\mu}=G^{\{\mu\}}$.
Each point in $G_{\mu}$ is called a \emph{generic point} of $\mu$.

We study DC1 in non-empty over-saturated sets for dynamical systems with the non-uniform specification property.

\begin{definition}[{\cite[Definition~1.1]{Lin-Tian-Yu-2024}}]
	We say that a dynamical system $(X,f)$ satisfies the \emph{non-uniform specification property} with \emph{gap function} $M(n,\varepsilon):\mathbb{N}\times(0,+\infty)\rightarrow\mathbb{N}$ if the following conditions hold:
	\begin{itemize}
		\item $M(n,\varepsilon)$ is non-decreasing with $n$ and non-increasing with $\varepsilon$;
		\item for any positive integer $k\geq2$, any $x_1,\cdots,x_k\in X$ and any non-negative integers $a_1,b_1,\cdots,a_k,b_k$ with
		\begin{equation*}
			a_1\leq b_1<\cdots<a_k\leq b_k
		\end{equation*}
		and
		\begin{equation*}
			a_{i+1}-b_i\geq M(b_i-a_i+1,\varepsilon)\quad\text{for}\quad 1\leq i\leq k-1,
		\end{equation*}
		there exists a point $z\in X$ such that
		\begin{equation*}
			d(f^{n-a_i}x_i,f^nz)\leq\varepsilon,\quad\text{for}\quad a_i\leq n\leq b_i,1\leq i\leq k.
		\end{equation*}
	\end{itemize}
	In particular, if for every $\varepsilon>0$, the gap function $M(n,\varepsilon)=M(\varepsilon)$ is independent of $n$, then $(X,f)$ is said to satisfy the \emph{specification property} with gap function $M(\varepsilon)$.
\end{definition}

Recall that a point $x\in X$ is called transitive if $\omega(x,f)=X$, where $\omega(x,f)$ denotes the $\omega$-limit set of $x$.
The set of transitive points in $(X,f)$ is denoted by $\mathrm{Trans}(f)$.
Now we state our first main result, which establishes the distributional chaos of type 1 in any non-empty over-saturated set.

\begin{maintheorem}\label{thm-DC1-1}
	Suppose that $(X,f)$ satisfies the non-uniform specification property.
	Let $K\subset\mathcal{M}_f(X)$ be a non-empty subset.
	If $G^K\neq\varnothing$, then for any non-empty open set $U\subset X$, there exists an uncountable DC1-scrambled set $S^K$ contained in $G^K\cap U\cap\mathrm{Trans}(f)$.
\end{maintheorem}

According to Lemma~\ref{pre-saturated-lemma6}, under the non-uniform specification property, the completely-irregular set $\mathrm{CI}(f)$ contains the non-empty over-saturated set $G^{\mathcal{K}}$, where $\mathcal{K}=\bigcup_{x\in X}V_f(x)$.
From Theorem~\ref{thm-DC1-1}, it follows that

\begin{maincorollary}\label{cor-DC1}
	Suppose that $(X,f)$ satisfies the non-uniform specification property, then $\mathrm{CI}(f)$ contains an uncountable DC1-scrambled set.
	In particular, every non-empty $\varphi$-irregular set contains an uncountable DC1-scrambled set under the non-uniform specification property.
\end{maincorollary}

According to Proposition~\ref{pre-saturated-prop5}, under the non-uniform specification property, every non-empty over-saturated set contains $G^{\mathcal{K}}$.
Thus, to prove Theorem~\ref{thm-DC1-1}, it suffices to construct an uncountable DC1-scrambled set in $G^{\mathcal{K}}$.
This leads to the following question.

\begin{question}\label{question}
	Under non-uniform specification, is there an uncountable DC1-scrambled set contained in $G^K\setminus G^L$ whenever $G^K\setminus G^L\neq\varnothing$?
\end{question}

We recall the definition of the non-uniform $\tau$-specification property from \cite{Lin-Tian-Yu-2025+}.
For a dynamical system $(X,f)$ with non-uniform specification, let $\mathscr{G}(X,f)$ be the collection of all gap functions for which $(X,f)$ satisfies the non-uniform specification property.
Elements of $\mathscr{G}(X,f)$ are called \emph{admissible gap functions}.
For $M\in\mathscr{G}(X,f)$, let
\begin{equation*}
	\tau(M):=\sup_{\varepsilon>0}\liminf_{n\rightarrow\infty}\frac{M(n,\varepsilon)}{n}.
\end{equation*}
Define
\begin{equation*}
	\tau_*(X,f)=\inf_{M\in\mathscr{G}(X,f)}\tau(M).
\end{equation*}
We say that $(X,f)$ satisfies the \emph{non-uniform $\tau$-specification property} if $\tau_*(X,f)=\tau$ and this infimum is attained by an admissible gap function.
Equivalently, there is an $M\in\mathscr{G}(X,f)$ with $\tau(M)=\tau$, and every $M'\in\mathscr{G}(X,f)$ satisfies $\tau(M')\geq\tau$.

We answer Question~\ref{question} for the case $\tau<\infty$ when both $K=\{\mu\}$ and $L=\{\nu\}$ are singletons in $\mathcal{M}_f^{erg}(X)$.

\begin{maintheorem}\label{thm-DC1-2}
	Let $\tau\in[0,\infty)$.
	Suppose that $(X,f)$ satisfies the non-uniform $\tau$-specification property.
	Then for any two distinct ergodic measures $\mu,\nu\in\mathcal{M}_f^{erg}(X)$ and any non-empty open set $U\subset X$, there exists an uncountable DC1-scrambled set contained in $(G^{\mu}\setminus G^{\nu})\cap U\cap\mathrm{Trans}(f)$.
\end{maintheorem}

The following result shows that the assumption $\tau<\infty$ in Theorem~\ref{thm-DC1-2} cannot be removed.

\begin{maintheorem}\label{thm-example-1}
	There exists a dynamical system $(X,f)$ satisfying the non-uniform $\infty$-specification property for which there exist two distinct ergodic measures $\mu,\nu\in\mathcal{M}_f^{erg}(X)$ such that $G^\mu\setminus G^\nu$ has no DC1-scrambled pairs.
\end{maintheorem}

Define the \emph{total variation distance} $\rho_{\mathsf{tv}}$ on $\mathcal{M}(X)$ as
\begin{equation*}
	\rho_{\mathsf{tv}}(\mu,\nu)=\sup_{B\in\mathfrak{B}(X)}|\mu(B)-\nu(B)|.
\end{equation*}
Clearly $0\leq\rho_{\mathsf{tv}}(\mu,\nu)\leq 1$.
A pair of points $\{p_0,p_1\}\subset X$ is called \emph{distal} if
\begin{equation*}
	\inf_{n\geq0} d(f^np_0,f^np_1)>0.
\end{equation*}

The following result, which is one of the main results, answers Question~\ref{question} under some additional conditions.
We will provide an example that satisfies these conditions in Section~\ref{Sect-example}.

\begin{maintheorem}\label{thm-DC1-3}
	Let $\tau\in[0,\infty)$.
	Suppose that $(X,f)$ satisfies the non-uniform $\tau$-specification property.
	Let $K,L$ be two (disjoint) non-empty closed sets in $\mathcal{M}_f(X)$ and $\{p_0,p_1\}$ a distal pair of $X$ satisfying the following conditions:
	\begin{enumerate}
		\item $K$ can be written as a union $K=\bigcup_{i\in I}K_i$, where each $K_i$ is a non-empty connected closed set with $G_{K_i}\neq\varnothing$;
		\item for any $\mu\in\overline{\mathrm{co}}(K\cup V_f(p_0)\cup V_f(p_1))$ and any $\nu\in L$, one has
		\begin{equation*}
			\rho_{\mathsf{tv}}(\mu,\nu)>\frac{\tau}{1+\tau}.
		\end{equation*}
	\end{enumerate}
	Then for any non-empty open set $U\subset X$, there exists an uncountable DC1-scrambled set contained in
	\begin{equation*}
		G^K\cap U\cap\mathrm{Trans}(f)\setminus\bigcup_{\nu\in L}G^{\nu}.
	\end{equation*}
\end{maintheorem}

\begin{remark}
	Using Theorem~\ref{thm-DC1-3}, one can easily see that the conclusion of Theorem~\ref{thm-DC1-2} can be extended to any $m+n$ distinct ergodic measures $\mu_1,\cdots,\mu_m,\nu_1,\cdots,\nu_n$.
	See Theorem~\ref{DC1-thm2}.
\end{remark}

The following example shows that the bound $\tau/(1+\tau)$ in condition~(2) of Theorem~\ref{thm-DC1-3} is optimal.

\begin{maintheorem}\label{thm-example-2}
	For any $0<\tau<\infty$, there exists a dynamical system $(X,f)$ with the non-uniform $\tau$-specification property such that
	\begin{enumerate}
		\renewcommand{\labelenumi}{(\theenumi)}
		\item there exist two disjoint non-empty compact subsets $K,L\subset\mathcal{M}_f(X)$ and a distal pair $\{p_0,p_1\}\subset X$ satisfying condition~(1) of Theorem~\ref{thm-DC1-3};
		\item for any $\mu\in\overline{\mathrm{co}}(K\cup V_f(p_0)\cup V_f(p_1))$ and $\nu\in L$ such that
		\begin{equation*}
			\rho_{\mathsf{tv}}(\mu,\nu)\geq\frac{\tau}{1+\tau},
		\end{equation*}
		and moreover, equality holds for some $(\mu,\nu)\in \overline{\mathrm{co}}(K\cup V_f(p_0)\cup V_f(p_1))\times L$;
		\item $G^K\setminus\bigcup_{\nu\in L}G^{\nu}$ has no DC1-scrambled pairs.
	\end{enumerate}
\end{maintheorem}

\subsection*{Organization of this paper}

In Section~\ref{Sect-pre}, we will introduce some basic notions and preliminary results.
In Section~\ref{Sect-distal}, we will find distal pairs under non-uniform specification, which will be used in the proofs of the main results.
In Section~\ref{Sect-DC1}, we will prove Theorem~\ref{thm-DC1-1}, Corollary~\ref{cor-DC1}, Theorem~\ref{thm-DC1-2} and Theorem~\ref{thm-DC1-3}.
In Section~\ref{Sect-example}, we will prove Theorem~\ref{thm-example-1} and Theorem~\ref{thm-example-2}, and provide an example with non-uniform specification which satisfies conditions of Theorem~\ref{thm-DC1-3}.

\section{Preliminaries}\label{Sect-pre}

\subsection{The first Wasserstein metric}

Let $\rho$ be the first Wasserstein metric on $\mathcal{M}(X)$, which metricizes the weak$^\ast$ topology on $\mathcal{M}(X)$, see \cite{Villani-2009} for more information.
According to \cite[p.~95, (6.3)]{Villani-2009}, for any $\mu,\nu\in\mathcal{M}(X)$, $\rho(\mu,\nu)$ can be represented by
\begin{equation}\label{eq-pre-metric-1}
	\rho(\mu,\nu)=\sup_{\varphi\in\mathsf{Lip}^1(X)}\left|\int_X\varphi\mathrm{d}\mu-\int_X\varphi\mathrm{d}\nu\right|,
\end{equation}
where $\mathsf{Lip}^1(X)$ is the space of real-valued Lipschitz functions on $X$ with Lipschitz constant at most $1$.
Then for any $x,y\in X$, we have
\begin{equation}\label{eq-pre-metric-2}
	\rho(\delta_x,\delta_y)=d(x,y).
\end{equation}
Let $\mathcal{B}(\mu,\varepsilon)$ denote the open ball in $\mathcal{M}(X)$ with respect to the metric $\rho$.

The following proposition can be easily checked by using \eqref{eq-pre-metric-1}, hence we omit the proof.

\begin{proposition}\label{pre-metric-prop1}
	Let $\mu,\mu_1,\cdots,\mu_n,\nu_1,\cdots,\nu_n\in\mathcal{M}(X)$ and $0\leq s_1,\cdots,s_n\leq 1$ with $s_1+\cdots+s_n=1$, then
	\begin{equation}\label{eq-pre-metric-prop1-1}
		\rho\bigg(\mu,\sum_{i=1}^ns_i\mu_i\bigg)\leq\sum_{i=1}^ns_i\rho(\mu,\mu_i)
	\end{equation}
	and
	\begin{equation}\label{eq-pre-metric-prop1-2}
		\rho\bigg(\sum_{i=1}^ns_i\mu_i,\sum_{i=1}^ns_i\nu_i\bigg)\leq\sum_{i=1}^ns_i\rho(\mu_i,\nu_i).
	\end{equation}
\end{proposition}

\begin{lemma}\cite[Lemma~2.2]{Lin-Tian-Yu-2024}\label{pre-metric-lemma2}
	Given $0<\varepsilon,\delta\leq 1$.
	Let $\{x_i\}_{i=0}^{n-1},\{y_i\}_{i=0}^{n-1}\subset X$.
	If $|\{i\in[0,n-1]:d(x_i,y_i)\leq\varepsilon\}|\geq(1-\delta)n$, then
	\begin{equation}\label{eq-pre-metric-lemma2}
		\rho\bigg(\frac1n\sum_{i=0}^{n-1}\delta_{x_i},\frac1n\sum_{i=0}^{n-1}\delta_{y_i}\bigg)\leq\varepsilon+\delta\cdot\mathrm{diam}(X),
	\end{equation}
	where $\mathrm{diam}(X)=\sup_{x,y\in X}d(x,y)$ is the diameter of $(X,d)$.
\end{lemma}

\subsection{Over-saturated sets}

In this subsection, we will provide some basic properties of over-saturated sets for dynamical systems with the non-uniform specification property.

%lemma1
\begin{lemma}[{\cite[Lemma~2.3]{Lin-Tian-Yu-2025+}}]\label{pre-saturated-lemma1}
	For any non-empty subset $K\subset\mathcal{M}_f(X)$, there exists a sequence $\{\nu_j\}_{j=1}^\infty$ contained in $K$ such that
	\begin{equation}\label{eq-pre-saturated-lemma1}
		\overline{\{\nu_j:j\geq n\}}=\overline{K},\quad\forall n\in\mathbb{N}.
	\end{equation}
	Moreover, if a sequence of invariant measures $\{\nu_j\}_{j=1}^{\infty}$ satisfies \eqref{eq-pre-saturated-lemma1}, then
	\begin{equation*}
		G^K=\bigcap_{j\geq1}G^{\nu_j}.
	\end{equation*}
\end{lemma}

\begin{lemma}[{\cite[Theorem~3.2 (4)]{Lin-Tian-Yu-2024}}]\label{pre-saturated-lemma2}
	Suppose that $(X,f)$ satisfies the non-uniform specification property.
	Let $K\subset\mathcal{M}_f(X)$ be a non-empty subset.
	If $G^K\neq\varnothing$, then $G^K$ is residual in $X$.
\end{lemma}

%prop3
\begin{proposition}[{\cite[Proposition~2.5]{Lin-Tian-Yu-2025+}}]\label{pre-saturated-prop3}
	Suppose that $(X,f)$ satisfies the non-uniform specification property.
	Let $K\subset\mathcal{M}_f(X)$ be a non-empty subset and $\{\nu_j\}_{j=1}^\infty$ be a sequence contained in $K$ satisfying
	\begin{equation*}
		\overline{\{\nu_j:j\geq n\}}=\overline{K},\quad\forall n\in\mathbb{N}.
	\end{equation*}
	Then $G^K\neq\varnothing$ if and only if $G^{\nu_j}\neq\varnothing$ for all $j\in\mathbb{N}$.
\end{proposition}

%\begin{proof}
%	The necessity follows from Lemma~\ref{pre-saturated-lemma1}.
%	For the sufficiency, by Lemma~\ref{pre-saturated-lemma2}, every $G^{\nu_j}$ is residual in $X$.
%	Using Lemma~\ref{pre-saturated-lemma1}, we conclude that $G^K=\bigcap_{j\geq1}G^{\nu_j}$ is residual in $X$.
%\end{proof}

%prop5
\begin{proposition}[{\cite[Proposition~2.7]{Lin-Tian-Yu-2025+}}]\label{pre-saturated-prop5}
	Suppose that $(X,f)$ satisfies the non-uniform specification.
	Let
	\begin{equation}\label{eq-pre-saturated-prop5}
		\mathcal{K}=\bigcup_{x\in X}V_f(x).
	\end{equation}
	Then
	\begin{enumerate}
		\item $G^K\neq\varnothing$ if and only if $K\in\mathcal{P}_*(\mathcal{K})$, where $\mathcal{P}_*(\mathcal{K})$ is the collection of non-empty subsets of $\mathcal{K}$;
		\item if $G^K\neq\varnothing$, then $G^{\mathcal{K}}\subset G^K$;
		\item $G_{\mathcal{K}}=G^{\mathcal{K}}\neq\varnothing$, and consequently $\mathcal{K}$ is a non-empty compact connected set.
	\end{enumerate}
\end{proposition}

%\begin{proof}
%	If $G^K\neq\varnothing$, then $K\subset V_f(x)\subset\mathcal{K}$ for any $x\in G^K$.
%	For the sufficiency, using Proposition~\ref{pre-saturated-prop3}, we obtain that $G^{\mathcal{K}}\neq\varnothing$, which implies that for any $K\in\mathcal{P}_*(\mathcal{K})$, $G^K\neq\varnothing$ since $G^K\supset G^{\mathcal{K}}$.
%	This completes the proof of statement (1).
%	The statement (2) follows immediately from (1).
%	
%	Now we show statement (3).
%	Clearly $G_{\mathcal{K}}\subset G^{\mathcal{K}}$.
%	On the other hand, for any $y\in G^{\mathcal{K}}$, we have
%	\begin{equation*}
%		\mathcal{K}\subset V_f(y)\subset\bigcup_{x\in X}V_f(x)=\mathcal{K},
%	\end{equation*}
%	which implies that $V_f(y)=\mathcal{K}$.
%	We conclude that $G_{\mathcal{K}}=G^{\mathcal{K}}\neq\varnothing$.
%	It follows from \cite[Proposition~3.8]{DGS1976} that $\mathcal{K}$ is a non-empty compact connected set since $\mathcal{K}=V_f(x)$ for any $x\in G_{\mathcal{K}}$.
%\end{proof}

%lemma6
\begin{lemma}[{\cite[Lemma~2.8]{Lin-Tian-Yu-2025+}}]\label{pre-saturated-lemma6}
	Suppose that $(X,f)$ satisfies the non-uniform specification.
	Then
	\begin{enumerate}
		\item for any $\varphi\in\mathfrak{C}(f)$, there exists $K\subset \mathcal{M}_f(X)$ such that $G^K\neq\varnothing$ and $G^K\subset I_{\varphi}(f)$;
		\item $G^{\mathcal{K}}\subset\mathrm{CI}(f)$.
	\end{enumerate}
\end{lemma}

\subsection{Total variation distance}

Define the \emph{total variation distance} $\rho_{\mathsf{tv}}$ on $\mathcal{M}(X)$ as
\begin{equation*}
	\rho_{\mathsf{tv}}(\mu,\nu)=\sup_{B\in\mathfrak{B}(X)}|\mu(B)-\nu(B)|.
\end{equation*}

\begin{proposition}\label{pre-tv-prop1}
	Let $\mu_1,\cdots,\mu_m$ and $\nu_1,\cdots,\nu_n$ be $m+n$ distinct measures in $\mathcal{M}(X)$ such that
	\begin{equation*}
		\rho_{\mathsf{tv}}(\mu_i,\nu_j)=1,\quad\forall\,1\leq i\leq m,1\leq j\leq n.
	\end{equation*}
	Then
	\begin{equation*}
		\rho_{\mathsf{tv}}(\mu,\nu)=1,\quad\forall\,\mu\in\mathrm{co}\{\mu_1,\cdots,\mu_m\},\nu\in\mathrm{co}\{\nu_1,\cdots,\nu_n\}.
	\end{equation*}
\end{proposition}

\begin{proof}
	Given $\varepsilon>0$, for any $1\leq i\leq m$ and any $1\leq j\leq n$, choose Borel measurable set $B_{i,j}$ such that
	\begin{equation*}
		|\mu_i(B_{i,j})-\nu_j(B_{i,j})|>1-\frac{\varepsilon}{2(m+n)}.
	\end{equation*}
	Replacing $B_{i,j}$ by its complement if necessary, we may assume that
	\begin{equation*}
		\mu_i(B_{i,j})-\nu_j(B_{i,j})>1-\frac{\varepsilon}{2(m+n)}.
	\end{equation*}
	Then
	\begin{equation*}
		\mu_i(B_{i,j})>1-\frac{\varepsilon}{2(m+n)}\quad\text{and}\quad\nu_j(B_{i,j})<\frac{\varepsilon}{2(m+n)}.
	\end{equation*}
	Write
	\begin{equation*}
		\mu=\sum_{i=1}^m\alpha_i\mu_i\quad\text{and}\quad \nu=\sum_{j=1}^n\beta_j\nu_j.
	\end{equation*}
	Let
	\begin{equation*}
		B=\bigcup_{i=1}^m\bigcap_{j=1}^nB_{i,j}.
	\end{equation*}
	Then we have
	\begin{equation*}
		\mu(B)=\sum_{i=1}^m\alpha_i\mu_i(B)\geq\sum_{i=1}^m\alpha_i\mu_i\Big(\bigcap_{j=1}^nB_{i,j}\Big)>1-\frac{\varepsilon}{2}
	\end{equation*}
	and
	\begin{equation*}
		\nu(B)=\sum_{j=1}^n\beta_j\nu_j(B)\leq\sum_{j=1}^n\beta_j\nu_j\Big(\bigcup_{i=1}^mB_{i,j}\Big)<\frac{\varepsilon}{2}.
	\end{equation*}
	Hence $\mu(B)-\nu(B)>1-\varepsilon$, which implies that
	\begin{equation*}
		\rho_{\mathsf{tv}}(\mu,\nu)>1-\varepsilon.
	\end{equation*}
	It follows from the arbitrariness of $\varepsilon$ that $\rho_{\mathsf{tv}}(\mu,\nu)=1$.
	This completes the proof.
\end{proof}

\subsection{$\sigma$-controlled function}

%def1
\begin{definition}
	Given $0<\sigma<1$, a continuous function $\varphi$ is called a \emph{$\sigma$-controlled function} if
	\begin{equation*}
		\sup_{\mu\in\mathcal{M}_f(X)}\int_X\varphi\mathrm{d}\mu-\inf_{\mu\in\mathcal{M}_f(X)}\int_X\varphi\mathrm{d}\mu>\sigma\Big(\sup_{x\in X}\varphi(x)-\inf_{x\in X}\varphi(x)\Big).
	\end{equation*}
	The set of $\sigma$-controlled functions is denoted by $C_{\sigma}(X)$.
\end{definition}

\begin{lemma}\label{pre-controlled-lemma2}
	Let $0<\sigma<1$ and $\mu,\nu\in\mathcal{M}_f(X)$.
	If $\rho_{\mathsf{tv}}(\mu,\nu)>\sigma$, then there exists $\varphi\in C_{\sigma}(X)$ such that
	\begin{equation*}
		\int_X\varphi\mathrm{d}\nu-\int_X\varphi\mathrm{d}\mu>\sigma\Big(\sup_{x\in X}\varphi(x)-\inf_{x\in X}\varphi(x)\Big).
	\end{equation*}
\end{lemma}

\begin{proof}
	Since $\rho_{\mathsf{tv}}(\mu,\nu)>\sigma$, we can choose $\kappa>0$ sufficiently small and take a Borel set $A$ such that
	\begin{equation}\label{eq-pre-controlled-lemma2-proof-1}
		\nu(A)-\mu(A)>\sigma+2\kappa.
	\end{equation}
	Since $\mu$ and $\nu$ are regular, we can choose two closed sets $B\subset A$ and $C\subset X\setminus A$ such that
	\begin{equation}\label{eq-pre-controlled-lemma2-proof-2}
		\nu(B)>\nu(A)-\kappa\quad\text{and}\quad\mu(X\setminus C)<\mu(A)+\kappa.
	\end{equation}
	Note that $B$ and $C$ are disjoint.
	By Urysohn's lemma, there exists $\varphi\in C(X)$ such that
	\begin{equation}\label{eq-pre-controlled-lemma2-proof-3}
		\varphi|_{B}\equiv 1,\;\varphi|_{C}\equiv 0\;\;\text{and}\;\;0\leq \varphi(x)\leq 1,\;\forall x\in X.
	\end{equation}
	Combining \eqref{eq-pre-controlled-lemma2-proof-1}, \eqref{eq-pre-controlled-lemma2-proof-2} and \eqref{eq-pre-controlled-lemma2-proof-3}, we conclude that
	\begin{equation*}
		\int_X\varphi\mathrm{d}\nu-\int_X\varphi\mathrm{d}\mu>\nu(B)-\mu(X\setminus C)>\nu(A)-\mu(A)-2\kappa>\sigma=\sigma\Big(\sup_{x\in X}\varphi(x)-\inf_{x\in X}\varphi(x)\Big).
	\end{equation*}
	This completes the proof.
\end{proof}

%\begin{corollary}\label{pre-controlled-cor3}
%	Let $\mu_1,\cdots,\mu_m$ and $\nu_1,\cdots,\nu_n$ be $m+n$ distinct invariant measures such that
%	\begin{equation*}
%		\rho_{\mathsf{tv}}(\mu_i,\nu_j)=1,\quad\forall\,1\leq i\leq m,1\leq j\leq n.
%	\end{equation*}
%	Then for any $0<\sigma<1$, there exists $\varphi\in C_{\sigma}(X)$ such that
%	\begin{equation*}
%		\int_X\varphi\mathrm{d}\nu-\int_X\varphi\mathrm{d}\mu>\sigma\Big(\sup_{x\in X}\varphi(x)-\inf_{x\in X}\varphi(x)\Big),\quad\forall\,\mu\in\mathrm{co}\{\mu_1,\cdots,\mu_m\},\nu\in\mathrm{co}\{\nu_1,\cdots,\nu_n\}.
%	\end{equation*}
%\end{corollary}

\section{Distal pairs}\label{Sect-distal}

In this section, we show that for dynamical systems with non-uniform specification, there are uncountably many ergodic measures whose generic points admit distal pairs.
Recall that a pair of points $\{x,y\}$ is called \emph{distal} if
\begin{equation*}
	\inf_{n\geq0} d(f^nx,f^ny)>0.
\end{equation*}
A pair of points $\{x,y\}$ is called \emph{proximal} if it is not distal.
Clearly $\{x,y\}$ is proximal if and only if
\begin{equation*}
	\liminf_{n\rightarrow\infty}d(f^nx,f^ny)=0.
\end{equation*}
A dynamical system $(X,f)$ is called \emph{proximal} if it has no distal pair.

A closed $f$-invariant set $M\subset X$ is called \emph{minimal} if $M$ is the unique non-empty closed $f$-invariant subset of $M$.
If a minimal set $M$ has at least two points, then we say $M$ is \emph{non-degenerate}.
The following lemma provides a necessary and sufficient condition for systems to be proximal.

\begin{lemma}[{\cite[Lemma~18]{Oprocha-2009}}]\label{distal-lemma1}
	$(X,f)$ is proximal if and only if it has a fixed point which is the unique minimal set.
	In particular, if $(X,f)$ has a non-degenerate minimal set, then $(X,f)$ has a distal pair.
\end{lemma}

Define the \emph{support} of $\mu\in\mathcal{M}(X)$ by
\begin{equation*}
	\mathrm{supp}(\mu)=\{x\in X:\mu(U)>0\;\text{for any neighborhood $U$ of $x$}\}.
\end{equation*}
The following lemma provides a sufficient condition for the existence of distal pairs in the set of generic points with respect to an ergodic measure.

\begin{lemma}[{\cite[Lemma~4.1]{Chen-Tian-2021}}]\label{distal-lemma2}
	Let $\mu$ be an ergodic measure with non-degenerate and minimal support, that is, $\mathrm{supp}(\mu)$ is a non-degenerate minimal set.
	Then $G_{\mu}$ contains a distal pair.
\end{lemma}

Recall that we always assume that $X$ has at least two points.

\begin{lemma}\label{distal-lemma3}
	Every dynamical system with non-uniform specification has a distal pair.
\end{lemma}

\begin{proof}
	Since $|X|\geq 2$, we can choose two distinct points $y,z\in X$ and $\varepsilon>0$ such that $d(y,z)>5\varepsilon$.
	Let $m=M(1,\varepsilon)$.
	Choose $\kappa>0$ such that for any $x_1,x_2\in X$ with $d(x_1,x_2)<\kappa$, we have
	\begin{equation}\label{eq-DC1-lemma3-proof-1}
		\max_{0\leq i\leq 3m-1}d(f^ix_1,f^ix_2)<\varepsilon.
	\end{equation}
	For $n\in\mathbb{N}$, define
	\begin{equation*}
		E_n=\bigcap_{i=0}^{n-1}f^{-2im}\Big(\overline{B(y,\varepsilon)}\cap f^{-m}\big(\overline{B(z,\varepsilon)}\big)\Big)
		\quad\text{and}\quad E=\bigcap_{n\geq1}E_n.
	\end{equation*}
	By non-uniform specification, each $E_n$ is a non-empty closed set.
	Hence $E$ is a non-empty closed set.

	We proceed to show that $\omega(x,f)$ contains a non-degenerate minimal set for all $x\in E$.
	Given $x\in E$, it suffices to show that $\omega(x,f)$ has no fixed point.
	Otherwise, there is a fixed point $x_0\in\omega(x,f)$.
	Then there exists a sufficiently large $k\in\mathbb{N}$ such that $f^kx\in B(x_0,\kappa)$.
	Write $k=2qm-r$, where $q\in\mathbb{N}$ and $0\leq r\leq 2m-1$.
	Using \eqref{eq-DC1-lemma3-proof-1}, we obtain that $f^{2qm}x=f^r(f^kx)\in B(x_0,\varepsilon)$ and $f^{2qm+m}x=f^{r+m}(f^kx)\in B(x_0,\varepsilon)$.
	By the definition of $E$, we have $f^{2qm}x\in\overline{B(y,\varepsilon)}$ and $f^{2qm+m}x\in\overline{B(z,\varepsilon)}$.
	Hence
	\begin{equation*}
		d(y,z)\leq d(y,f^{2qm}x)+d(f^{2qm}x,x_0)+d(x_0,f^{2qm+m}x)+d(f^{2qm+m}x,z)<4\varepsilon,
	\end{equation*}
	which contradicts $d(y,z)>5\varepsilon$.
	We conclude that $\omega(x,f)$ contains a non-degenerate minimal set and thus $(X,f)$ has a distal pair.
\end{proof}

Consider the one-sided full shift over two symbols $(\Sigma,\sigma)$, where $\Sigma=\{0,1\}^{\mathbb{N}_0}$.
For every irrational number $\alpha\in(0,1)$, let $\Sigma_\alpha\subset\Sigma$ be the Sturmian subshift of slope $\alpha$, obtained by coding the irrational rotation
\begin{equation*}
	R_\alpha(t)=t+\alpha\pmod 1
\end{equation*}
with respect to the partition $[0,1-\alpha)$ and $[1-\alpha,1)$.
It is well known that $\Sigma_\alpha$ is an infinite minimal subshift and is uniquely ergodic, and that its unique invariant measure $\eta_\alpha$ satisfies
\begin{equation*}
	\eta_\alpha([1])=\alpha,
\end{equation*}
where $[1]=\{\omega\in\Sigma:\omega_0=1\}$.
If $\Sigma_\alpha\cap\Sigma_\beta\neq\varnothing$, then the minimality of these two subshifts implies that $\Sigma_\alpha=\Sigma_\beta$.
Their unique invariant measures are therefore equal, and hence
\begin{equation*}
	\alpha=\eta_\alpha([1])=\eta_\beta([1])=\beta.
\end{equation*}
Consequently, the family
\begin{equation*}
	\{\Sigma_\alpha:\alpha\in(0,1)\setminus\mathbb{Q}\}
\end{equation*}
consists of uncountably many pairwise disjoint infinite minimal subshifts.
For more information on Sturmian sequences and the associated one-sided shift spaces, see \cite[\S~2.1]{Berthe-Holton-Zamboni-2006}.

\begin{proposition}\label{distal-prop4}
	Given a dynamical system $(X,f)$, let $N$ be a positive integer with $N\geq 2$ and $F$ be an $f^N$-invariant closed set.
	Assume that
	\begin{enumerate}
		\item $F$ can be written as a union of pairwise disjoint non-empty closed sets:
		\begin{equation*}
			F=\bigcup_{\omega\in\Sigma}F(\omega),
		\end{equation*}
		where $\Sigma:=\{0,1\}^{\mathbb{N}_0}$.
		\item The map $\pi:F\to\Sigma$ defined by
		\begin{equation*}
			x\in F(\omega)\mapsto\omega,\quad\forall\omega\in\Sigma,
		\end{equation*}
		is continuous.
		\item The shift action $\sigma:\Sigma\to\Sigma$ satisfies
		\begin{equation*}
			\pi\circ f^N=\sigma\circ\pi.
		\end{equation*}
	\end{enumerate}
	Then there exist uncountably many ergodic measures with non-degenerate minimal support in $\mathcal{M}_f^{erg}(Y)$, where $Y\subset X$ is the closed $f$-invariant set defined by
	\begin{equation*}
		Y=\bigcup_{n=0}^{N-1}f^n(F).
	\end{equation*}
\end{proposition}

\begin{proof}
	By the preceding discussion on Sturmian subshifts, there exist an uncountable index set $I$ and minimal sequences $\omega^{(i)}\in\Sigma$, $i\in I$, such that the sets
	\begin{equation*}
		\overline{\mathrm{Orb}(\omega^{(i)},\sigma)},\quad i\in I,
	\end{equation*}
	are pairwise disjoint infinite minimal subshifts of $\Sigma$.
	For each $i\in I$, the set
	\begin{equation*}
		\pi^{-1}\big(\overline{\mathrm{Orb}(\omega^{(i)},\sigma)}\big)=\bigcup_{\xi\in\overline{\mathrm{Orb}(\omega^{(i)},\sigma)}}F(\xi)
	\end{equation*}
	is a non-empty closed $f^N$-invariant subset of $F$.
	Choose an $f^N$-minimal set
	\begin{equation*}
		M_i\subset\pi^{-1}\big(\overline{\mathrm{Orb}(\omega^{(i)},\sigma)}\big).
	\end{equation*}
	Since $\pi(M_i)$ is a non-empty closed $\sigma$-invariant subset of the minimal subshift $\overline{\mathrm{Orb}(\omega^{(i)},\sigma)}$, we have
	\begin{equation*}
		\pi(M_i)=\overline{\mathrm{Orb}(\omega^{(i)},\sigma)}.
	\end{equation*}
	In particular, $M_i$ is infinite.
	Moreover, the pairwise disjointness of the subshifts $\overline{\mathrm{Orb}(\omega^{(i)},\sigma)}$ implies that the sets $M_i$, $i\in I$, are pairwise disjoint.

	For every $i\in I$, define
	\begin{equation*}
		Y_i=\bigcup_{n=0}^{N-1}f^n(M_i)\subset Y.
	\end{equation*}
	Since $M_i$ is $f^N$-minimal, we have $f^N(M_i)=M_i$.
	Hence $Y_i$ is a closed $f$-invariant set.
	For any $x\in M_i$ and any $0\leq r\leq N-1$, the $f^N$-orbit of $x$ is dense in $M_i$, and therefore
	\begin{equation*}
		\overline{\mathrm{Orb}(f^rx,f)}=\bigcup_{n=0}^{N-1}f^n(M_i)=Y_i.
	\end{equation*}
	Thus $Y_i$ is an infinite, and hence non-degenerate, $f$-minimal set.

	We next observe that every $f$-minimal set contains at most $N$ distinct $f^N$-minimal subsets.
	Indeed, let $M$ be an $f$-minimal set and let $L\subset M$ be an $f^N$-minimal set.
	Then
	\begin{equation*}
		\bigcup_{n=0}^{N-1}f^n(L)
	\end{equation*}
	is a non-empty closed $f$-invariant subset of $M$, and hence
	\begin{equation*}
		M=\bigcup_{n=0}^{N-1}f^n(L).
	\end{equation*}
	Each $f^n(L)$ is $f^N$-minimal.
	If $L'\subset M$ is any $f^N$-minimal set, then $L'$ intersects some $f^n(L)$ and consequently $L'=f^n(L)$.
	This proves the observation.

	It follows that, for each fixed $f$-minimal set $M\subset Y$, there are at most $N$ indices $i\in I$ such that $Y_i=M$.
	Therefore the family $\{Y_i:i\in I\}$ contains uncountably many distinct $f$-minimal sets.
	Choose an uncountable subset $I'\subset I$ such that $Y_i\neq Y_j$ whenever $i,j\in I'$ and $i\neq j$.
	Since two distinct minimal sets are disjoint, the sets $Y_i$, $i\in I'$, are pairwise disjoint.

	For every $i\in I'$, choose an ergodic measure $\lambda_i\in\mathcal{M}_f^{erg}(Y_i)$.
	The minimality of $Y_i$ implies that
	\begin{equation*}
		\mathrm{supp}(\lambda_i)=Y_i.
	\end{equation*}
	Thus $\{\lambda_i:i\in I'\}$ is an uncountable family of pairwise distinct ergodic measures with non-degenerate minimal support contained in $\mathcal{M}_f^{erg}(Y)$.
	This completes the proof.
\end{proof}

\begin{theorem}\label{distal-thm5}
	Suppose that $(X,f)$ satisfies the non-uniform specification property with gap function $M(n,\varepsilon)$.
	Then there exist uncountably many ergodic measures with non-degenerate minimal support.
	Moreover, if
	\begin{equation*}
		\tau:=\sup_{\varepsilon>0}\liminf_{n\rightarrow\infty}\frac{M(n,\varepsilon)}{n}<\infty,
	\end{equation*}
	then for any $\mu\in\mathcal{M}_f^{erg}(X)$ and any $\varepsilon>0$, there exist uncountably many ergodic measures with non-degenerate minimal support contained in
	\begin{equation*}
		\mathcal{B}\Big(\mu,\frac{\tau}{1+\tau}\mathrm{diam}(X)+\varepsilon\Big).
	\end{equation*}
	As a result,
	\begin{equation*}
		\{\nu\in\mathcal{M}_f^{erg}(X):G_{\nu}\;\text{contains a distal pair}\}\cap\mathcal{B}\Big(\mu,\frac{\tau}{1+\tau}\mathrm{diam}(X)+\varepsilon\Big)
	\end{equation*}
	contains uncountably many elements.
\end{theorem}

\begin{proof}
	Given $\mu\in\mathcal{M}_f^{erg}(X)$ and $\varepsilon>0$, we aim to find positive integer $N\geq 2$ and closed $f^N$-invariant subset $F\subset X$ such that conditions~(1)--(3) in Proposition~\ref{distal-prop4} hold.
	Moreover, if $\tau<\infty$, we will show that the closed $f$-invariant set $Y$ defined by $Y:=\bigcup_{n=0}^{N-1}f^n(F)$ satisfies
	\begin{equation*}
		\mathcal{M}_f(Y)\subset\mathcal{B}\Big(\mu,\frac{\tau}{1+\tau}\mathrm{diam}(X)+\varepsilon\Big).
	\end{equation*}
	Then the conclusion of Theorem~\ref{distal-thm5} holds by applying Proposition~\ref{distal-prop4} and Lemma~\ref{distal-lemma2}.

	Since $\mu$ is ergodic, we can choose a point $y\in G_\mu$.
	Fix two distinct points $z_0,z_1\in X$ and take $0<\delta<\min\{\varepsilon/8,d(z_0,z_1)/3\}$.
	Let $m=M(1,\delta)$.
	By non-uniform specification, for any $l\in\mathbb{N}$ and each $s=0,1$, there exists $y_s^{(l)}\in X$ such that
	\begin{equation*}
		d(y_s^{(l)},z_s)\leq\delta \quad\text{and}\quad d(f^{m+i}y_s^{(l)},f^iy)\leq\delta,\qquad \forall0\leq i\leq l-1.
	\end{equation*}
	For each $s=0,1$, let $y_s$ be an accumulation point of $\{y_s^{(l)}:l\in\mathbb{N}\}$.
	Then
	\begin{equation*}
		d(y_s,z_s)\leq\delta \quad\text{and}\quad d(f^{m+i}y_s,f^iy)\leq\delta,\qquad \forall i\in\mathbb{N}_0.
	\end{equation*}
	Since $d(y_0,y_1)\geq d(z_0,z_1)-2\delta>0$, $y_0$ and $y_1$ are distinct.
	By Proposition~\ref{pre-metric-prop1}, for any $n>m$ and $s=0,1$, we have
	\begin{equation*}
		\rho(\delta_{y_s}^n,\mu)\leq\frac{m}{n}\operatorname{diam}(X)+\delta+\rho(\delta_y^{n-m},\mu).
	\end{equation*}
	Since $y\in G_\mu$ and $\delta<\varepsilon/8$, we can choose $N_0\in\mathbb{N}$ sufficiently large such that
	\begin{equation}\label{eq-distal-thm5-proof-1}
		\delta_{y_s}^n\in\mathcal{B}(\mu,\varepsilon/4),\quad\forall\,n\geq N_0,\,s=0,1.
	\end{equation}
	Take $\varepsilon_0>0$ sufficiently small such that $\varepsilon_0<\min\{1/2,d(y_0,y_1)/3,\varepsilon/4\}$.
	If $\tau=\infty$, then we set $n_0=N_0$ and $m_0=M(n_0,\varepsilon_0)$.
	If $\tau<\infty$, then we can choose $\theta>0$ sufficiently small such that
	\begin{equation}\label{eq-distal-thm5-proof-2}
		\frac{\tau+2\theta}{1+\tau+2\theta}\mathrm{diam}(X)<\frac{\tau}{1+\tau}\mathrm{diam}(X)+\frac{\varepsilon}{4}.
	\end{equation}
	Moreover, since gap function $M(n,\varepsilon')$ is non-increasing in $\varepsilon'$, we may replace $\varepsilon_0$ by a smaller one such that for any $0<\varepsilon'\leq\varepsilon_0$,
	\begin{equation*}
		\bigg|\liminf_{n\to\infty}\frac{M(n,\varepsilon')}{n}-\tau\bigg|<\theta.
	\end{equation*}
	Then we choose $n_0\geq N_0$ sufficiently large such that
	\begin{equation*}
		\bigg|\frac{M(n_0,\varepsilon_0)}{n_0}-\tau\bigg|<2\theta
	\end{equation*}
	and let $m_0=M(n_0,\varepsilon_0)$.
	Then
	\begin{equation}\label{eq-distal-thm5-proof-3}
		(\tau-2\theta)n_0\leq m_0\leq (\tau+2\theta)n_0.
	\end{equation}
	For every non-negative integer $j$, let $a(j)=j(n_0+m_0)$ and $b(j)=j(n_0+m_0)+n_0-1$.

	Let $\Sigma=\{0,1\}^{\mathbb{N}_0}$.
	For any $\omega=(\omega_0,\omega_1,\cdots)\in\Sigma$ and any $k\in\mathbb{N}$, define
	\begin{equation*}
		F_k(\omega)=\bigcap_{j=0}^{k-1}\bigcap_{i=a(j)}^{b(j)}f^{-i}\Big(\overline{B(f^{i-a(j)}y_{\omega_j},\varepsilon_0)}\Big).
	\end{equation*}
	From the non-uniform specification property, it follows that each $F_k(\omega)$ is a non-empty closed set.
	Similarly, for any $l\in\mathbb{N}$, we can define $F_k(\zeta)$ for $\zeta\in\{0,1\}^l$ and $1\leq k\leq l$.
	For any $l\geq 1$, define
	\begin{equation*}
		F_l=\bigcup_{\zeta\in\{0,1\}^l}F_l(\zeta)
	\end{equation*}
	and $F=\bigcap_{l\geq1}F_l$.
	Then each $F_l$ is non-empty and closed, satisfying $F_{l+1}\subset F_l$.
	Hence $F$ is a non-empty closed set.
	Define
	\begin{equation}
		F(\omega)=\bigcap_{k\geq1}F_k(\omega).
	\end{equation}
	Clearly, $\{F_k(\omega)\}_{k=1}^{\infty}$ is a decreasing sequence of closed sets.
	Hence for any $\omega\in\Sigma$, $F(\omega)$ is a non-empty closed set.
	It is easy to see that
	\begin{equation}
		F=\bigcup_{\omega\in\Sigma}F(\omega).
	\end{equation}
	Moreover, since $\varepsilon_0<d(y_0,y_1)/3$, one can easily see that $\{F(\omega):\omega\in\Sigma\}$ are pairwise disjoint.
	From the construction of $F(\omega)$ and $F_k(\omega)$, for any $\omega\in\Sigma$ and any $i\in\mathbb{N}_0$, we have $f^{i(n_0+m_0)}(F(\omega))\subset F(\sigma^i\omega)$, where $\sigma:\Sigma\to\Sigma$ is the shift map.
	Let $N=n_0+m_0$.
	Then $F$ is a closed $f^N$-invariant set.
	Hence conditions~(1) and~(3) of Proposition~\ref{distal-prop4} hold.

	We proceed to show (2).
	For any $l\in\mathbb{N}$ and $\zeta\in\{0,1\}^l$, let
	\begin{equation*}
		[\zeta]=\{\omega\in\Sigma:\omega_i=\zeta_i,\;\forall\,0\leq i\leq l-1\}.
	\end{equation*}
	Since $\varepsilon_0<d(y_0,y_1)/3$, the sets $F_l(\xi)$, $\xi\in\{0,1\}^l$, are pairwise disjoint.
	Hence
	\begin{equation*}
		\pi^{-1}([\zeta])=F\cap F_l(\zeta).
	\end{equation*}
	Since $F_l(\zeta)$ is closed and the complement of $F\cap F_l(\zeta)$ in $F$ is a finite union of closed sets, $\pi^{-1}([\zeta])$ is open in $F$.
	Since the cylinder sets form a basis of $\Sigma$, we conclude that $\pi:x\in F(\omega)\mapsto \omega$ is continuous.
	Hence condition~(2) of Proposition~\ref{distal-prop4} holds.

	Let $Y:=\bigcup_{n=0}^{N-1}f^n(F)$, then $Y$ is a closed $f$-invariant subset.
	By Proposition~\ref{distal-prop4}, $\mathcal{M}_f(Y)$ contains uncountably many ergodic measures with non-degenerate minimal support.
	This completes the proof of Theorem~\ref{distal-thm5} for the case $\tau=\infty$.
	It remains to show that for the case $\tau<\infty$,
	\begin{equation*}
		\mathcal{M}_f(Y)\subset\mathcal{B}\Big(\mu,\frac{\tau}{1+\tau}\mathrm{diam}(X)+\varepsilon\Big).
	\end{equation*}
	Given $\nu\in\mathcal{M}_f^{erg}(Y)$, there exists $0\leq r\leq N-1$ such that $\nu(f^r(F))>0$.
	Hence we can choose $x\in G_{\nu}\cap f^r(F)$.
	Since every forward iterate of a $\nu$-generic point is still $\nu$-generic, replacing $x$ by $f^{N-r}x$, we may assume that $x\in G_\nu\cap F$.
	Take $\omega\in\Sigma$ with $x\in F(\omega)$.
	Then by the construction of $F(\omega)$, Proposition~\ref{pre-metric-prop1}, \eqref{eq-distal-thm5-proof-1}, \eqref{eq-distal-thm5-proof-2} and \eqref{eq-distal-thm5-proof-3}, for any $k\in\mathbb{N}$, we obtain that
	\begin{equation*}
		\begin{split}
			\rho(\delta_x^{kN},\mu)&\leq\frac{n_0}{k(n_0+m_0)}\sum_{j=0}^{k-1}\rho\bigg(\frac{1}{n_0}\sum_{i=a(j)}^{b(j)}\delta_{f^i(x)},\mu\bigg)+\frac{m_0}{n_0+m_0}\mathrm{diam}(X)\\
			&\leq\frac{n_0}{k(n_0+m_0)}\sum_{j=0}^{k-1}\rho\bigg(\frac{1}{n_0}\sum_{i=a(j)}^{b(j)}\delta_{f^{i-a(j)}(y_{\omega_j})},\mu\bigg)+\varepsilon_0+\frac{m_0}{n_0+m_0}\mathrm{diam}(X)\\
			&<\frac{\varepsilon}{4}+\frac{\varepsilon}{4}+\frac{\tau+2\theta}{1+\tau+2\theta}\mathrm{diam}(X)\leq\frac{\tau}{1+\tau}\mathrm{diam}(X)+\frac{3}{4}\varepsilon.
		\end{split}
	\end{equation*}
	This implies that
	\begin{equation*}
		\rho(\nu,\mu)\leq\frac{\tau}{1+\tau}\mathrm{diam}(X)+\frac{3}{4}\varepsilon.
	\end{equation*}
	We conclude that
	\begin{equation*}
		\mathcal{M}_f^{erg}(Y)\subset\overline{\mathcal{B}\Big(\mu,\frac{\tau}{1+\tau}\mathrm{diam}(X)+\frac{3}{4}\varepsilon\Big)}.
	\end{equation*}
	Indeed, for any $\lambda\in\mathcal{M}_f(Y)$, let $\kappa_\lambda$ be its ergodic decomposition.
	Then
	\begin{equation*}
		\rho(\lambda,\mu)\leq\int_{\mathcal{M}_f^{erg}(Y)}\rho(\eta,\mu)\,\mathrm{d}\kappa_\lambda(\eta)\leq\frac{\tau}{1+\tau}\mathrm{diam}(X)+\frac{3}{4}\varepsilon.
	\end{equation*}
	It follows from the ergodic decomposition theorem that
	\begin{equation*}
		\mathcal{M}_f(Y)\subset\overline{\mathcal{B}\Big(\mu,\frac{\tau}{1+\tau}\mathrm{diam}(X)+\frac{3}{4}\varepsilon\Big)}\subset\mathcal{B}\Big(\mu,\frac{\tau}{1+\tau}\mathrm{diam}(X)+\varepsilon\Big).
	\end{equation*}
	This completes the proof.
\end{proof}

\begin{corollary}
	Suppose that $(X,f)$ satisfies the specification property.
	Then for any $\varepsilon>0$ and any invariant measure $\mu\in\mathcal{M}_f(X)$, $\mathcal{B}(\mu,\varepsilon)$ contains uncountably many ergodic measures with non-degenerate minimal support.
\end{corollary}

\begin{proof}
	It follows from \cite[Propositions~21.3 and~21.8]{DGS1976} that, for every dynamical system with the specification property, ergodic measures are dense in the space of invariant measures.
	Moreover, specification with gap function $M(\varepsilon)$ implies non-uniform specification with gap function $M(n,\varepsilon):=M(\varepsilon)$.
	In particular,
	\begin{equation*}
		\tau=\sup_{\varepsilon'>0}\liminf_{n\to\infty}\frac{M(n,\varepsilon')}{n}=0.
	\end{equation*}
	Given $\mu\in\mathcal{M}_f(X)$ and $\varepsilon>0$, take $\mu_0\in\mathcal{M}_f^{erg}(X)$ such that $\rho(\mu,\mu_0)<\varepsilon/2$.
	Then $\mathcal{B}(\mu_0,\varepsilon/2)\subset\mathcal{B}(\mu,\varepsilon)$.
	By Theorem~\ref{distal-thm5}, the ball $\mathcal{B}(\mu_0,\varepsilon/2)$ contains uncountably many ergodic measures with non-degenerate minimal support, and so does $\mathcal{B}(\mu,\varepsilon)$.
	This completes the proof.
\end{proof}

\section{Distributional chaos}\label{Sect-DC1}

In this section, we prove Theorem~\ref{thm-DC1-1}, Corollary~\ref{cor-DC1}, Theorem~\ref{thm-DC1-2} and Theorem~\ref{thm-DC1-3}.

\subsection{Proof of Theorem~\ref{thm-DC1-1} and Corollary~\ref{cor-DC1}}

According to Proposition~\ref{pre-saturated-prop5}, if $(X,f)$ satisfies the non-uniform specification property, then every non-empty over-saturated set contains $G_{\mathcal{K}}$, where $\mathcal{K}$ is the non-empty connected closed subset of $\mathcal{M}_f(X)$ defined by \eqref{eq-pre-saturated-prop5}.
Therefore Theorem~\ref{thm-DC1-1} is a direct consequence of the following theorem.

\begin{theorem}\label{DC1-thm1}
	Suppose that $(X,f)$ satisfies the non-uniform specification property.
	Then for any non-empty open set $U\subset X$, there exists an uncountable DC1-scrambled set $S\subset G_{\mathcal{K}}\cap U\cap\mathrm{Trans}(f)$.
\end{theorem}

\begin{proof}
	Let $M(n,\varepsilon)$ be an admissible gap function.
	Fix a distal pair $\{p_0,p_1\}$.
	Let
	\begin{equation}\label{eq-DC1-thm1-proof-1}
		\eta=\frac13\inf_{i\geq 0}d(f^ip_0,f^ip_1)>0.
	\end{equation}
	Choose a sequence $\{\nu_j\}_{j=1}^{\infty}$ in $\mathcal{K}$ satisfying
	\begin{equation}\label{eq-DC1-thm1-proof-2}
		\overline{\{\nu_j:j\geq n\}}=\mathcal{K},\quad\forall n\in\mathbb{N}.
	\end{equation}
	Fix $y\in G_{\mathcal{K}}$ and for each $j\geq1$, take an increasing sequence $\{n_j(l)\}_{l=1}^{\infty}$ such that
	\begin{equation}\label{eq-DC1-thm1-proof-3}
		\lim_{l\rightarrow\infty}\delta_y^{n_j(l)}=\nu_j.
	\end{equation}
	
	Arbitrarily given an open set $U\subset X$.
	According to \cite[Theorem~3.2]{Lin-Tian-Yu-2024}, $(X,f)$ is topologically transitive.
	Hence $\mathrm{Trans}(f)$ is residual in $X$.
	Therefore we can choose a $z\in\mathrm{Trans}(f)\cap U$.
	Arbitrarily given $0<\varepsilon<\eta$ such that $\overline{B(z,\varepsilon)}\subset U$, where $B(z,\varepsilon)$ denotes the open ball with center $z$ and radius $\varepsilon$.
	Let $\varepsilon_i=\varepsilon/2^i$ for $i\in\mathbb{Z}$.
	Now we define:
	\begin{itemize}
		\item sequences of non-negative integers $\{a_j\}_{j=1}^{\infty}$, $\{b_j\}_{j=1}^{\infty}$, and $\{c_j\}_{j=1}^{\infty}$;
		\item arrays of non-negative integers  $\{l_{jr}:j\in\mathbb{N},1\leq r\leq j\}$, $\{s_{jr}:j\in\mathbb{N},1\leq r\leq j\}$, and $\{t_{jr}:j\in\mathbb{N},1\leq r\leq j\}$.
	\end{itemize}
	
	Let $a_1=0$ and $b_1=a_1+M(1,\varepsilon_1)$.
	Choose $l_{11}\in\mathbb{N}$ such that $n_1(l_{11})>b_1^2$ and let
	\begin{equation*}
		c_1=b_1+n_1(l_{11}),\quad s_{11}=c_1+M(c_1-b_1+1,\varepsilon_1)\quad\text{and}\quad t_{11}=s_{11}+s_{11}^2.
	\end{equation*}
	Suppose that $\{a_j\}_{j=1}^k$, $\{b_j\}_{j=1}^k$, $\{c_j\}_{j=1}^k$, $\{l_{jr}:1\leq j\leq k,1\leq r\leq j\}$, $\{s_{jr}:1\leq j\leq k,1\leq r\leq j\}$ and $\{t_{jr}:1\leq j\leq k,1\leq r\leq j\}$ have been defined.
	Let
	\begin{equation*}
		\begin{split}
			a_{k+1}&=t_{kk}+M(t_{kk}+1,\varepsilon_{k+1}),\\
			b_{k+1}&=a_{k+1}+M(1,\varepsilon_{k+1}).
		\end{split}
	\end{equation*}
	Choose positive integers $l_{k+1,1}<l_{k+1,2}<\cdots<l_{k+1,k+1}$ such that
	\begin{equation*}
		b_{k+1}^2<n_1(l_{k+1,1})<n_2(l_{k+1,2})<\cdots<n_{k+1}(l_{k+1,k+1}).
	\end{equation*}
	Let
	\begin{equation*}
		\begin{split}
			c_{k+1}&=b_{k+1}+n_{k+1}(l_{k+1,k+1}),\\
			s_{k+1,1}&=c_{k+1}+M(c_{k+1}-b_{k+1}+1,\varepsilon_{k+1}),\\
			t_{k+1,1}&=s_{k+1,1}+s_{k+1,1}^2
		\end{split}
	\end{equation*}
	and for $j=1,2,\cdots,k$, we inductively define
	\begin{equation*}
		\begin{split}
			s_{k+1,j+1}&=t_{k+1,j}+M(t_{k+1,j}-s_{k+1,j}+1,\varepsilon_{k+1}),\\
			t_{k+1,j+1}&=s_{k+1,j+1}+s_{k+1,j+1}^2.
		\end{split}
	\end{equation*}
	Inductively, we can define $\{a_j\}_{j\in\mathbb{N}}$, $\{b_j\}_{j\in\mathbb{N}}$, $\{c_j\}_{j\in\mathbb{N}}$,
	$\{l_{jr}:j\in\mathbb{N},1\leq r\leq j\}$, $\{s_{jr}:j\in\mathbb{N},1\leq r\leq j\}$ and $\{t_{jr}:j\in\mathbb{N},1\leq r\leq j\}$.
	By the construction of these sequences and arrays, it is easy to see that for any $j\in\mathbb{N}$, we have
	\begin{equation*}
		a_j<b_j<c_j<s_{j1}<t_{j1}<\cdots<s_{jj}<t_{jj}<a_{j+1}.
	\end{equation*}
	
	Now we define $x_{\xi}\in X$ for each $\xi\in\{0,1\}^{\mathbb{N}}$ as the limit of a Cauchy sequence $\{x_{\xi_1\cdots\xi_k}\}_{k=1}^{\infty}$, constructed inductively as follows.
	By non-uniform specification, there is an $x_{\xi_1}$ such that
	\begin{equation*}
		\begin{split}
			d(f^ix_{\xi_1},z)\leq\varepsilon_1\quad&\text{for}\quad i=a_1,\\
			d(f^ix_{\xi_1},f^{i-b_1}y)\leq\varepsilon_1\quad&\text{for}\quad b_1\leq i\leq c_1,\\
			d(f^ix_{\xi_1},f^ip_{\xi_1})\leq\varepsilon_1\quad&\text{for}\quad s_{11}\leq i\leq t_{11}.
		\end{split}
	\end{equation*}
	
	Assume that $x_{\xi_1\cdots\xi_k}$ has been defined.
	We construct $x_{\xi_1\cdots\xi_{k+1}}$ as follows.
	By non-uniform specification, there is an $x_{\xi_1\cdots\xi_{k+1}}$ such that
	\begin{equation}\label{eq-DC1-thm1-proof-4}
		\begin{split}
			d(f^ix_{\xi_1\cdots\xi_{k+1}},f^ix_{\xi_1\cdots\xi_k})\leq\varepsilon_{k+1}&\quad\text{for}\quad 0\leq i\leq t_{kk},\\
			d(f^ix_{\xi_1\cdots\xi_{k+1}},z)\leq\varepsilon_{k+1}&\quad\text{for}\quad i=a_{k+1},\\
			d(f^ix_{\xi_1\cdots\xi_{k+1}},f^{i-b_{k+1}}y)\leq\varepsilon_{k+1}&\quad\text{for}\quad b_{k+1}\leq i\leq c_{k+1}
		\end{split}
	\end{equation}
	and for $r=1,2,\cdots,k+1$
	\begin{equation}\label{eq-DC1-thm1-proof-5}
		d(f^ix_{\xi_1\cdots\xi_{k+1}},f^ip_{\xi_r})\leq\varepsilon_{k+1}\quad\text{for}\quad s_{k+1,r}\leq i\leq t_{k+1,r}.
	\end{equation}
	Inductively, we construct a sequence $\{x_{\xi_1\cdots\xi_k}\}_{k=1}^{\infty}$.
	One can easily check by \eqref{eq-DC1-thm1-proof-4} that $\{x_{\xi_1\cdots\xi_k}\}_{k=1}^{\infty}$ is a Cauchy sequence in $\overline{B(z,\varepsilon)}$.
	Let $x_{\xi}$ be the limit point of $\{x_{\xi_1\cdots\xi_k}\}_{k=1}^{\infty}$.
	Then by \eqref{eq-DC1-thm1-proof-4}, for any $k\in\mathbb{N}$, we have
	\begin{equation}\label{eq-DC1-thm1-proof-6}
		d(f^ix_{\xi},f^ix_{\xi_1\cdots\xi_k})\leq\varepsilon_k,\quad\forall 0\leq i\leq t_{kk}.
	\end{equation}
	Combining \eqref{eq-DC1-thm1-proof-4}, \eqref{eq-DC1-thm1-proof-5} and \eqref{eq-DC1-thm1-proof-6}, for any $k\in\mathbb{N}$ and $1\leq j\leq k$, we obtain that
	\begin{equation}\label{eq-DC1-thm1-proof-7}
		\begin{split}
			d(f^ix_{\xi},z)\leq\varepsilon_{k-1}&\quad\text{for}\quad i=a_k,\\
			d(f^ix_{\xi},f^{i-b_k}y)\leq\varepsilon_{k-1}&\quad\text{for}\quad b_k\leq i\leq c_k,\\
			d(f^ix_{\xi},f^ip_{\xi_j})\leq\varepsilon_{k-1}&\quad\text{for}\quad s_{kj}\leq i\leq t_{kj}.
		\end{split}
	\end{equation}

	\begin{figure}[h]
		\centering
		\begin{tikzpicture}
			\draw [black, dashed] (0,2)--(1,2) (3,2)--(5,2) (8,2)--(13,2);
			\draw [thick, black] (1,2)--(3,2) (5,2)--(8,2);
			\draw [thick, black] (0,1)--(13,1);
			\fill (0,2) circle (1pt);
			\fill (1,2) circle (1pt);
			\fill (3,2) circle (1pt);
			\fill (5,2) circle (1pt);
			\fill (8,2) circle (1pt);
			\fill (12,2) circle (1pt);
			\fill (0,1) circle (1pt);
			\draw (0,1) node [below] {$f^{a_k}(x_{\xi})$};
			\draw (0,2) node [above] {$a_k$};
			\draw (1,2) node [above] {$b_k$};
			\draw (3,2) node [above] {$c_k$};
			\draw (5,2) node [above] {$s_{kj}$};
			\draw (8,2) node [above] {$t_{kj}$};
			\draw (12,2) node [above] {$a_{k+1}$};
			\draw (0,2) node [below] {$z$};
			\draw (12,2) node [below] {$z$};
			\draw (2,2) node [below] {$\underbrace{\hspace{2cm}}$};
			\draw (6.5,2) node [below] {$\underbrace{\hspace{3cm}}$};
			\draw (2,1.7) node [below] {$\{y,\cdots,f^{c_k-b_k}(y)\}$};
			\draw (6.5,1.7) node [below] {$\{f^{s_{kj}}(p_{\xi_j}),\cdots,f^{t_{kj}}(p_{\xi_j})\}$};
		\end{tikzpicture}
		\caption{Construction of $x_{\xi}$, where $z\in\mathrm{Trans}(f)$, $y\in G_{\mathcal{K}}$ and $\{p_0,p_1\}$ is a distal pair.}
	\end{figure}
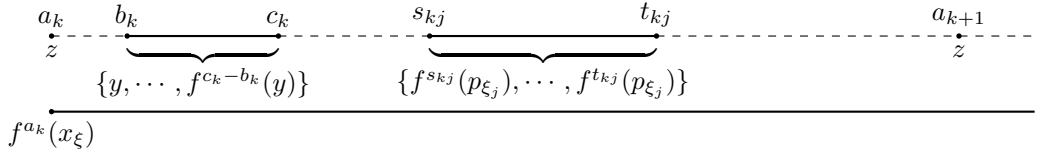

	Let $S=\{x_{\xi}:\xi\in\{0,1\}^{\mathbb{N}}\}$.
	We claim that $S\subset G_{\mathcal{K}}\cap U\cap\mathrm{Trans}(f)$.
	Clearly $S\subset\overline{B(z,\varepsilon)}\subset U$.
	We proceed to show that $S\subset\mathrm{Trans}(f)$.
	Since $z\in\mathrm{Trans}(f)$, it suffices to show that $z\in\omega(x_{\xi},f)$ holds for every $\xi\in\{0,1\}^{\mathbb{N}}$.
	Fix any $\xi\in\{0,1\}^{\mathbb{N}}$.
	By \eqref{eq-DC1-thm1-proof-7}, we have $f^{a_k}x_{\xi}\rightarrow z$ as $k\rightarrow\infty$.
	This implies that $z\in\omega(x_{\xi},f)$ and hence $S\subset\mathrm{Trans}(f)$.
	
	Now we show that $S\subset G_{\mathcal{K}}$.
	For any $r\geq 1$ and $j\geq r$, write $n_{jr}=n_r(l_{jr})$.
	By \eqref{eq-DC1-thm1-proof-3}, for any $r\in\mathbb{N}$, we have $\delta_y^{n_{jr}}\rightarrow\nu_r$ weakly as $j\rightarrow\infty$.
	By Lemma~\ref{pre-saturated-lemma1}, Proposition~\ref{pre-saturated-prop5}, and \eqref{eq-DC1-thm1-proof-2}, it suffices to show that $x_{\xi}\in G^{\nu_j}$ for every $\xi\in\{0,1\}^{\mathbb{N}}$ and $j\in\mathbb{N}$.
	Using Lemma~\ref{pre-metric-lemma2} and \eqref{eq-DC1-thm1-proof-7}, we obtain that
	\begin{equation}\label{eq-DC1-thm1-proof-8}
		\rho(\delta_{x_{\xi}}^{b_k+n_{kj}},\delta_y^{n_{kj}})\leq\varepsilon_{k-1}+\frac{b_k}{b_k+n_{kj}}\mathrm{diam}(X),\quad\forall k>j.
	\end{equation}
	Since $n_{kj}=n_j(l_{kj})>n_1(l_{k1})>b_k^2$, by \eqref{eq-DC1-thm1-proof-8}, we obtain that
	\begin{equation*}
		\lim_{k\rightarrow\infty}\delta_{x_{\xi}}^{b_k+n_{kj}}=\lim_{k\rightarrow\infty}\delta_y^{n_{kj}}=\nu_j.
	\end{equation*}
	Therefore $x_{\xi}\in G^{\nu_j}$ and we conclude that $S\subset G_{\mathcal{K}}\cap U\cap\mathrm{Trans}(f)$.
	
	It remains to show that $S$ is an uncountable DC1-scrambled set.
	Fix $\xi,\zeta\in\{0,1\}^{\mathbb{N}}$ with $\xi\neq\zeta$.
	We aim to show that $\{x_{\xi},x_{\zeta}\}$ is a DC1-scrambled pair.
	Assume that $\xi_m\neq\zeta_m$ for some $m\in\mathbb{N}$.
	We need to show that for any $t>0$, we have
	\begin{equation}\label{eq-DC1-thm1-proof-9}
		\limsup_{n\rightarrow\infty}\frac1n|\{i\in[0,n-1]:d(f^ix_\xi,f^ix_\zeta)<t\}|=1
	\end{equation}
	and there is a $t_0>0$ such that
	\begin{equation}\label{eq-DC1-thm1-proof-10}
		\liminf_{n\rightarrow\infty}\frac1n|\{i\in[0,n-1]:d(f^ix_\xi,f^ix_\zeta)<t_0\}|=0.
	\end{equation}
	For any $t>0$, choose $k>0$ sufficiently large such that $\varepsilon_k<t$.
	By \eqref{eq-DC1-thm1-proof-7}, we obtain that
	\begin{equation*}
		\{i\in\mathbb{N}_0:d(f^ix_\xi,f^ix_\zeta)<t\}\supset\bigcup_{j\geq k+2}([b_j,c_j]\cap\mathbb{N}).
	\end{equation*}
	Therefore
	\begin{equation*}
		\limsup_{n\rightarrow\infty}\frac1n|\{i\in[0,n-1]:d(f^ix_\xi,f^ix_\zeta)<t\}|\geq\lim_{j\rightarrow\infty}\frac{c_j-b_j+1}{c_j+1}\geq \lim_{j\rightarrow\infty}\frac{b_j^2+1}{b_j^2+b_j+1}=1,
	\end{equation*}
	which implies \eqref{eq-DC1-thm1-proof-9}.
	By \eqref{eq-DC1-thm1-proof-1}, we obtain that
	\begin{equation*}
		\{i\in\mathbb{N}_0:d(f^ix_\xi,f^ix_\zeta)\geq\eta\}\supset\bigcup_{j\geq m}([s_{jm},t_{jm}]\cap\mathbb{N}).
	\end{equation*}
	Therefore
	\begin{equation*}
		\limsup_{n\rightarrow\infty}\frac1n|\{i\in[0,n-1]:d(f^ix_\xi,f^ix_\zeta)\geq\eta\}|
		\geq\lim_{j\rightarrow\infty}\frac{t_{jm}-s_{jm}+1}{t_{jm}+1}
		=\lim_{j\rightarrow\infty}\frac{s_{jm}^2+1}{s_{jm}^2+s_{jm}+1}=1,
	\end{equation*}
	which implies \eqref{eq-DC1-thm1-proof-10}.
	Combining \eqref{eq-DC1-thm1-proof-9} and \eqref{eq-DC1-thm1-proof-10}, we conclude that $\{x_{\xi},x_{\zeta}\}$ is a DC1-scrambled pair.
	In particular, $x_{\xi}\neq x_{\zeta}$.
	Therefore the cardinality of $S$ is equal to $2^{\aleph_0}>\aleph_0$, which implies that $S$ is an uncountable DC1-scrambled set.
\end{proof}

\begin{proof}[{\bf Proof of Theorem~\ref{thm-DC1-1}}]
	According to Proposition~\ref{pre-saturated-prop5}, we have $G_{\mathcal{K}}\subset G^K$.
	From Theorem~\ref{DC1-thm1}, there exists an uncountable DC1-scrambled set $S^K$ contained in $G^K\cap U\cap\mathrm{Trans}(f)$.
\end{proof}

\begin{proof}[{\bf Proof of Corollary~\ref{cor-DC1}}]
	By Proposition~\ref{pre-saturated-prop5} and Lemma~\ref{pre-saturated-lemma6}, we have $G_{\mathcal{K}}\subset\mathrm{CI}(f)$.
	From Theorem~\ref{DC1-thm1}, it follows that $\mathrm{CI}(f)$ contains an uncountable DC1-scrambled set.
\end{proof}

\subsection{Proof of Theorem~\ref{thm-DC1-2} and Theorem~\ref{thm-DC1-3}}

We prove Theorem~\ref{thm-DC1-3} before the proof of Theorem~\ref{thm-DC1-2}.

\begin{proof}[\bf{Proof of Theorem~\ref{thm-DC1-3}}]
	Let $M(n,\varepsilon)$ be an admissible gap function with $\tau(M)=\tau$.
	Let $\widetilde{K}:=\overline{\mathrm{co}}(K\cup V_f(p_0)\cup V_f(p_1))$.
	If $\tau>0$, then by condition~(2) and Lemma~\ref{pre-controlled-lemma2}, for any $\mu\in\widetilde{K}$ and $\nu\in L$, there exists $\varphi_{\mu,\nu}\in C_{\frac{\tau}{1+\tau}}(X)$ such that
	\begin{equation*}
		\int_X\varphi_{\mu,\nu}\mathrm{d}\nu-\int_X\varphi_{\mu,\nu}\mathrm{d}\mu>\frac{\tau}{1+\tau}\Big(\sup_{x\in X}\varphi_{\mu,\nu}(x)-\inf_{x\in X}\varphi_{\mu,\nu}(x)\Big).
	\end{equation*}
	Clearly, when $\tau=0$, such $\varphi_{\mu,\nu}$ exists as well.
	Without loss of generality, we may assume that
	\begin{equation}\label{eq-thm-DC1-3-proof-1}
		\sup_{x\in X}\varphi_{\mu,\nu}(x)=1\;\text{and}\;\inf_{x\in X}\varphi_{\mu,\nu}(x)=0,\quad\forall\,(\mu,\nu)\in\widetilde{K}\times L.
	\end{equation}
	Then there exist two open neighborhoods $\mathcal{U}_{\mu,\nu}$ and $\mathcal{V}_{\mu,\nu}\subset\mathcal{M}(X)$ of $\mu$ and $\nu$, respectively, such that for any $(\lambda,\omega)\in \mathcal{U}_{\mu,\nu}\times \mathcal{V}_{\mu,\nu}$,
	\begin{equation*}
		\int_X\varphi_{\mu,\nu}\mathrm{d}\omega-\int_X\varphi_{\mu,\nu}\mathrm{d}\lambda>\frac{\tau}{1+\tau}.
	\end{equation*}
	By shrinking $\mathcal{U}_{\mu,\nu}$ and $\mathcal{V}_{\mu,\nu}$, we can choose $\eta_{\mu,\nu}>0$ such that
	\begin{equation*}
		\int_X\varphi_{\mu,\nu}\mathrm{d}\omega-\int_X\varphi_{\mu,\nu}\mathrm{d}\lambda>\frac{\tau}{1+\tau}+\eta_{\mu,\nu}.
	\end{equation*}
	Hence for any $(\lambda,\omega)\in\mathcal{U}_{\mu,\nu}\times\mathcal{V}_{\mu,\nu}$,
	\begin{equation}\label{eq-thm-DC1-3-proof-2}
		0\leq\int_X\varphi_{\mu,\nu}\mathrm{d}\lambda\leq\frac{1}{1+\tau}\quad\text{and}\quad\frac{\tau}{1+\tau}\leq\int_X\varphi_{\mu,\nu}\mathrm{d}\omega\leq1.
	\end{equation}
	Since $\widetilde{K}\times L$ is compact, there exist $\widetilde{\mu}_1,\cdots,\widetilde{\mu}_m\in\widetilde{K}$ and $\widetilde{\nu}_1,\cdots,\widetilde{\nu}_m\in L$ such that
	\begin{itemize}
		\item $\{\mathcal{U}_j\times \mathcal{V}_j:1\leq j\leq m\}$ covers $\widetilde{K}\times L$, where $\mathcal{U}_j:=\mathcal{U}_{\widetilde{\mu}_j,\widetilde{\nu}_j}$ and $\mathcal{V}_j:=\mathcal{V}_{\widetilde{\mu}_j,\widetilde{\nu}_j}$;
		\item for any $(\lambda,\omega)\in \mathcal{U}_j\times \mathcal{V}_j$,
		\begin{equation}\label{eq-thm-DC1-3-proof-3}
			\int_X\varphi_j\mathrm{d}\omega-\int_X\varphi_j\mathrm{d}\lambda>\frac{\tau}{1+\tau}+\eta_j,
		\end{equation}
		where $\varphi_j=\varphi_{\widetilde{\mu}_j,\widetilde{\nu}_j}$ and $\eta_j=\eta_{\widetilde{\mu}_j,\widetilde{\nu}_j}$.
	\end{itemize}

	Since $\widetilde{K}\times L$ is compact and each $\mathcal{U}_j\times\mathcal{V}_j$ is open in $\mathcal{M}(X)\times\mathcal{M}(X)$, we can choose $\kappa>0$ such that for any $(\mu,\nu)\in\widetilde{K}\times L$, there exists $1\leq j\leq m$ satisfying
	\begin{equation*}
		\mathcal{B}(\mu,\kappa)\times\mathcal{B}(\nu,\kappa)\subset\mathcal{U}_j\times\mathcal{V}_j.
	\end{equation*}
	Indeed, for each $(\mu,\nu)\in\widetilde{K}\times L$, choose $1\leq j(\mu,\nu)\leq m$ and $r_{\mu,\nu}>0$ such that
	\begin{equation*}
		\mathcal{B}(\mu,2r_{\mu,\nu})\times\mathcal{B}(\nu,2r_{\mu,\nu})\subset\mathcal{U}_{j(\mu,\nu)}\times\mathcal{V}_{j(\mu,\nu)}.
	\end{equation*}
	Taking a finite subcover of $\widetilde{K}\times L$ by the sets $\mathcal{B}(\mu,r_{\mu,\nu})\times\mathcal{B}(\nu,r_{\mu,\nu})$ and letting $\kappa$ be the minimum of the corresponding radii, we obtain the required property.

	By \eqref{eq-thm-DC1-3-proof-3}, we can take $\eta>0$ sufficiently small such that for any $1\leq j\leq m$ and any $(\lambda,\omega)\in \mathcal{U}_j\times \mathcal{V}_j$,
	\begin{equation}\label{eq-thm-DC1-3-proof-4}
		\int_X\varphi_j\mathrm{d}\omega-\int_X\varphi_j\mathrm{d}\lambda>\frac{\tau}{1+\tau}+3(1+\tau)\eta.
	\end{equation}
	Consider continuous function $\Theta(\vartheta):[0,1/5]\rightarrow\mathbb{R}$ defined as
	\begin{equation*}
		\Theta(\vartheta)=-\frac{(7\tau+4)\vartheta}{(1+\tau-2\vartheta)(1+\tau)}+\frac{3(1+\vartheta)(1+\tau)\eta}{1+\tau-2\vartheta}.
	\end{equation*}
	Then $\Theta(0)=3\eta>0$.
	Hence we can take $\theta>0$ sufficiently small such that $\Theta(\theta)>2\eta$.
	In particular, if $\tau>0$, then we choose $\theta$ sufficiently small such that $\theta<\tau/3$.
	Combining \eqref{eq-thm-DC1-3-proof-2} and \eqref{eq-thm-DC1-3-proof-4}, for any $1\leq r\leq m$ and any $(\widehat{\mu},\nu)\in\mathcal{U}_r\times\mathcal{V}_r$, when $\tau>0$, we obtain that
	\begin{equation*}
		\begin{split}
			&\frac{1+\theta}{1+\tau-2\theta}\int_X\varphi_r\mathrm{d}\widehat{\mu}+\eta+\frac{\tau+4\theta}{1+\tau-2\theta}\\
			\leq&\frac{1+\theta}{1+\tau-2\theta}\bigg(\int_X\varphi_r\mathrm{d}\nu-\frac{\tau}{1+\tau}-3(1+\tau)\eta\bigg)+\eta+\frac{\tau+4\theta}{1+\tau-2\theta}\\
			\leq&\int_X\varphi_r\mathrm{d}\nu-\frac{\tau-3\theta}{1+\tau-2\theta}\int_X\varphi_r\mathrm{d}\nu+\eta-\frac{3(1+\theta)(1+\tau)\eta}{1+\tau-2\theta}-\frac{1}{1+\tau-2\theta}\frac{\tau}{1+\tau}+\frac{\tau+4\theta}{1+\tau-2\theta}\\
			\leq&\int_X\varphi_r\mathrm{d}\nu+\eta-\frac{\tau-3\theta}{1+\tau-2\theta}\frac{\tau}{1+\tau}-\frac{3(1+\theta)(1+\tau)\eta}{1+\tau-2\theta}-\frac{1}{1+\tau-2\theta}\frac{\tau}{1+\tau}+\frac{\tau+4\theta}{1+\tau-2\theta}\\
			=&\int_X\varphi_r\mathrm{d}\nu+\eta+\frac{(7\tau+4)\theta}{(1+\tau-2\theta)(1+\tau)}-\frac{3(1+\theta)(1+\tau)\eta}{1+\tau-2\theta}\\
			=&\int_X\varphi_r\mathrm{d}\nu+\eta-\Theta(\theta)<\int_X\varphi_r\mathrm{d}\nu-\eta.
		\end{split}
	\end{equation*}
	We conclude that
	\begin{equation}\label{eq-thm-DC1-3-proof-5}
		\frac{1+\theta}{1+\tau-2\theta}\int_X\varphi_r\mathrm{d}\widehat{\mu}+\eta+\frac{\tau+4\theta}{1+\tau-2\theta}<\int_X\varphi_r\mathrm{d}\nu-\eta.
	\end{equation}
	When $\tau=0$, it follows directly from \eqref{eq-thm-DC1-3-proof-4} that \eqref{eq-thm-DC1-3-proof-5} holds for $\theta>0$ sufficiently small.

	Choose $\{\mu_k\}_{k=1}^{\infty}\subset K$ such that
	\begin{equation}\label{eq-thm-DC1-3-proof-6}
		\overline{\{\mu_k:k\geq n\}}=K,\quad\forall\, n\in\mathbb{N}.
	\end{equation}
	For any $k\geq 1$, choose $i(k)\in I$ with $\mu_k\in K_{i(k)}$ and fix $y_k\in G_{K_{i(k)}}$.
	For any $k\in\mathbb{N}$, there is an increasing sequence of positive integers $\{n_k(l)\}_{l=1}^{\infty}$ such that
	\begin{equation}\label{eq-thm-DC1-3-proof-7}
		\rho(\delta_{y_k}^{n_k(l)},\mu_k)\leq 2^{-l},\quad\forall\,l\in\mathbb{N}.
	\end{equation}
	Let
	\begin{equation*}
		\gamma=\frac13\inf_{n\geq0}d(f^np_0,f^np_1)>0.
	\end{equation*}
	Take $N\in\mathbb{N}$ sufficiently large such that for any $n\geq N$,
	\begin{equation}\label{eq-thm-DC1-3-proof-8}
		\delta_{p_0}^n\in\mathcal{B}(V_f(p_0),\kappa/3)\subset\mathcal{B}(\widetilde{K},\kappa/3)\quad\text{and}\quad\delta_{p_1}^n\in\mathcal{B}(V_f(p_1),\kappa/3)\subset\mathcal{B}(\widetilde{K},\kappa/3).
	\end{equation}
	Choose an increasing sequence of positive integers $\{N_k\}_{k\geq1}$ such that $N_1>N$ and for any $k\in\mathbb{N}$,
	\begin{equation}\label{eq-thm-DC1-3-proof-9}
		\delta_{y_k}^n\in\mathcal{B}(K_{i(k)},\kappa/3)\subset\mathcal{B}(\widetilde{K},\kappa/3),\quad\forall\, n\geq N_k.
	\end{equation}
	Recall that $\{p_0,p_1\}$ is a distal pair.
	Fix $z_0\in U$ and $\zeta>0$ sufficiently small such that $\overline{B(z_0,3\zeta)}\subset U$.
	Take $\varepsilon_0>0$ sufficiently small such that
	\begin{itemize}
		\item $\varepsilon_0<\min\{1/2,\gamma,\zeta,\kappa/3\}$;
		\item for any $1\leq j\leq m$, $|\varphi_j(x)-\varphi_j(y)|<\eta$ whenever $d(x,y)<\varepsilon_0$;
		\item for any $0<\varepsilon<\varepsilon_0$,
		\begin{equation*}
			\bigg|\liminf_{n\rightarrow\infty}\frac{M(n,\varepsilon)}{n}-\tau\bigg|<\theta.
		\end{equation*}
	\end{itemize}
	For any $k\geq1$, let $\varepsilon_k=\varepsilon_0/2^k$ and fix $z_{k,1},\cdots,z_{k,m_k}$ such that $\{B(z_{kl},\varepsilon_k):1\leq l\leq m_k\}$ covers $X$.

	We proceed to define six arrays of positive integers $\{a_{kr}:k\in\mathbb{N},1\leq r\leq k\}$, $\{b_{kr}:k\in\mathbb{N},1\leq r\leq k\}$, $\{q_{kr}:k\in\mathbb{N},1\leq r\leq m_k\}$, $\{l_{kr}:k\in\mathbb{N},1\leq r\leq k\}$, $\{s_{kr}:k\in\mathbb{N},1\leq r\leq k\}$ and $\{t_{kr}:k\in\mathbb{N},1\leq r\leq k\}$.
	Let $a_0=b_0=0$ and choose $a_{11}>\theta^{-1}N+b_0+M(1,\varepsilon_1)$.
	Take $b_{11}>a_{11}^2+a_{11}+N+\theta^{-1}N_1$ sufficiently large such that
	\begin{equation*}
		\bigg|\frac{M(b_{11}-a_{11}+1,\varepsilon_1)+m_1\cdot M(1,\varepsilon_1)}{b_{11}-a_{11}}-\tau\bigg|<2\theta.
	\end{equation*}
	Then
	\begin{equation*}
		\frac{N}{b_{11}-a_{11}}<\frac{a_{11}}{b_{11}-a_{11}}<\frac{1}{a_{11}}<\theta.
	\end{equation*}
	For each $1\leq r\leq m_1$, let $q_{1,r}=b_{11}+M(b_{11}-a_{11}+1,\varepsilon_1)+(r-1)M(1,\varepsilon_1)$.
	Let $s_{11}=q_{1,m_1}+M(1,\varepsilon_1)$, then $s_{11}>b_{11}>\theta^{-1}N_1$.
	Choose $l_{11}\in\mathbb{N}$ sufficiently large such that $n_1(l_{11})>s_{11}^2$.
	Take $t_{11}\in\mathbb{N}$ sufficiently large such that $t_{11}>s_{11}+n_1(l_{11})>s_{11}^2+s_{11}$ and
	\begin{equation*}
		\bigg|\frac{M(t_{11}-s_{11}+1,\varepsilon_1)}{t_{11}-s_{11}}-\tau\bigg|<2\theta\quad\text{and}\quad\bigg|\frac{M(t_{11}+1,\varepsilon_2)}{t_{11}-s_{11}}-\tau\bigg|<2\theta.
	\end{equation*}
	Then
	\begin{equation*}
		\frac{N_1}{t_{11}-s_{11}}<\frac{s_{11}}{t_{11}-s_{11}}<\frac{1}{s_{11}}<\theta.
	\end{equation*}
	Assume that $\{a_{jr}:1\leq j\leq k,1\leq r\leq j\}$, $\{b_{jr}:1\leq j\leq k,1\leq r\leq j\}$, $\{q_{jr}:1\leq j\leq k,1\leq r\leq m_j\}$, $\{l_{jr}:1\leq j\leq k,1\leq r\leq j\}$, $\{s_{jr}:1\leq j\leq k,1\leq r\leq j\}$ and $\{t_{jr}:1\leq j\leq k,1\leq r\leq j\}$ have been defined.
	Let $a_{k+1,1}=t_{kk}+M(t_{kk}+1,\varepsilon_{k+1})$.
	For $1\leq i\leq k$, we inductively define $b_{k+1,i}$ and $a_{k+1,i+1}$ as follows.
	If $a_{k+1,i}$ has been taken, then we take $b_{k+1,i}>a_{k+1,i}^2+a_{k+1,i}+N$ sufficiently large such that
	\begin{equation*}
		\bigg|\frac{M(b_{k+1,i}-a_{k+1,i}+1,\varepsilon_{k+1})}{b_{k+1,i}-a_{k+1,i}}-\tau\bigg|<2\theta.
	\end{equation*}
	Let $a_{k+1,i+1}=b_{k+1,i}+M(b_{k+1,i}-a_{k+1,i}+1,\varepsilon_{k+1})$.
	For $i=k+1$, we take $b_{k+1,k+1}>a_{k+1,k+1}^2+a_{k+1,k+1}+N+\theta^{-1}N_{k+1}$ sufficiently large such that
	\begin{equation*}
		\bigg|\frac{M(b_{k+1,k+1}-a_{k+1,k+1}+1,\varepsilon_{k+1})+m_{k+1}\cdot M(1,\varepsilon_{k+1})}{b_{k+1,k+1}-a_{k+1,k+1}}-\tau\bigg|<2\theta.
	\end{equation*}
	Then for any $1\leq j\leq k+1$,
	\begin{equation*}
		\frac{N}{b_{k+1,j}-a_{k+1,j}}<\frac{a_{k+1,j}}{b_{k+1,j}-a_{k+1,j}}<\frac{1}{a_{k+1,j}}<\theta.
	\end{equation*}
	For each $1\leq r\leq m_{k+1}$, let $q_{k+1,r}=b_{k+1,k+1}+M(b_{k+1,k+1}-a_{k+1,k+1}+1,\varepsilon_{k+1})+(r-1)M(1,\varepsilon_{k+1})$.
	Let $s_{k+1,1}=q_{k+1,m_{k+1}}+M(1,\varepsilon_{k+1})$, then $s_{k+1,1}>b_{k+1,k+1}>\theta^{-1}N_{k+1}\geq\theta^{-1}N_j$ for all $1\leq j\leq k+1$.
	Choose $l_{k+1,1}\in\mathbb{N}$ sufficiently large such that $n_1(l_{k+1,1})>s_{k+1,1}^2$.
	Take $t_{k+1,1}\in\mathbb{N}$ sufficiently large such that $t_{k+1,1}>s_{k+1,1}+n_1(l_{k+1,1})>s_{k+1,1}^2+s_{k+1,1}$ and
	\begin{equation*}
		\bigg|\frac{M(t_{k+1,1}-s_{k+1,1}+1,\varepsilon_{k+1})}{t_{k+1,1}-s_{k+1,1}}-\tau\bigg|<2\theta.
	\end{equation*}
	Then
	\begin{equation*}
		\frac{N_1}{t_{k+1,1}-s_{k+1,1}}<\frac{s_{k+1,1}}{t_{k+1,1}-s_{k+1,1}}<\frac{1}{s_{k+1,1}}<\theta.
	\end{equation*}
	Similarly we can take positive integers $l_{k+1,j},s_{k+1,j},t_{k+1,j}$ for $2\leq j\leq k+1$ inductively such that
	\begin{itemize}
		\item $s_{k+1,j}=t_{k+1,j-1}+M(t_{k+1,j-1}-s_{k+1,j-1}+1,\varepsilon_{k+1})>\theta^{-1}N_j$,
		\item $n_j(l_{k+1,j})>s_{k+1,j}^2$,
		\item $t_{k+1,j}>s_{k+1,j}+n_j(l_{k+1,j})>s_{k+1,j}^2+s_{k+1,j}$,
	\end{itemize}
	and
	\begin{equation}\label{eq-thm-DC1-3-proof-10}
		\bigg|\frac{M(t_{k+1,j}-s_{k+1,j}+1,\varepsilon_{k+1})}{t_{k+1,j}-s_{k+1,j}}-\tau\bigg|<2\theta.
	\end{equation}
	Then
	\begin{equation*}
		\frac{N_j}{t_{k+1,j}-s_{k+1,j}}<\frac{s_{k+1,j}}{t_{k+1,j}-s_{k+1,j}}<\frac{1}{s_{k+1,j}}<\theta.
	\end{equation*}
	Moreover, $t_{k+1,k+1}$ can be taken such that
	\begin{equation}\label{eq-thm-DC1-3-proof-11}
		\bigg|\frac{M(t_{k+1,k+1}+1,\varepsilon_{k+2})}{t_{k+1,k+1}-s_{k+1,k+1}}-\tau\bigg|<2\theta.
	\end{equation}
	Indeed, we can take an increasing sequence $\{\tilde{t}_l\}_{l=1}^{\infty}$ such that
	\begin{equation*}
		\lim_{l\to\infty}\frac{M(\tilde{t}_l,\varepsilon_{k+2})}{\tilde{t}_l}=\liminf_{n\to\infty}\frac{M(n,\varepsilon_{k+2})}{n}.
	\end{equation*}
	Since $s_{k+1,k+1}$ is fixed and
	\begin{equation*}
		M(\tilde{t}_l-s_{k+1,k+1},\varepsilon_{k+1})
		\leq M(\tilde{t}_l,\varepsilon_{k+2}),
	\end{equation*}
	we can take $l$ sufficiently large and let $t_{k+1,k+1}=\tilde{t}_l-1$ such that \eqref{eq-thm-DC1-3-proof-10} and \eqref{eq-thm-DC1-3-proof-11} hold.
	Inductively, we can define $\{a_{kr}:k\in\mathbb{N},1\leq r\leq k\}$, $\{b_{kr}:k\in\mathbb{N},1\leq r\leq k\}$, $\{q_{kr}:k\in\mathbb{N},1\leq r\leq m_k\}$, $\{l_{kr}:k\in\mathbb{N},1\leq r\leq k\}$, $\{s_{kr}:k\in\mathbb{N},1\leq r\leq k\}$ and $\{t_{kr}:k\in\mathbb{N},1\leq r\leq k\}$.
	For convenience, let
	\begin{equation*}
		a_{k,k+1}=s_{k1}\quad\text{and}\quad s_{k,k+1}=a_{k+1,1},\quad\forall\,k\in\mathbb{N}.
	\end{equation*}
	By the construction of these arrays, for any $k\in\mathbb{N}$, it is clear that
	\begin{equation}\label{eq-thm-DC1-3-proof-12}
		\begin{split}
			b_{kj}>a_{kj}^2+a_{kj},\;\text{and}\; t_{kj}>s_{kj}+n_j(l_{kj})>s_{kj}^2+s_{kj},&\quad\forall 1\leq j\leq k;\\
			\frac{N}{b_{kj}-a_{kj}}<\frac{a_{kj}}{b_{kj}-a_{kj}}<\frac{1}{a_{kj}}<\theta,&\quad\forall 1\leq j\leq k;\\
			\frac{N_j}{t_{kj}-s_{kj}}<\frac{s_{kj}}{t_{kj}-s_{kj}}<\frac{1}{s_{kj}}<\theta,&\quad\forall 1\leq j\leq k;\\
			\bigg|\frac{a_{k,j+1}-b_{kj}}{b_{kj}-a_{kj}}-\tau\bigg|<2\theta,&\quad\forall 1\leq j\leq k;\\
			\bigg|\frac{s_{k,j+1}-t_{kj}}{t_{kj}-s_{kj}}-\tau\bigg|<2\theta,&\quad\forall 1\leq j\leq k.
		\end{split}
	\end{equation}

	Now we define $x_{\xi}\in X$ for each $\xi\in\{0,1\}^{\mathbb{N}}$ as the limit of a Cauchy sequence $\{x_{\xi_1\cdots\xi_k}\}_{k=1}^{\infty}$, constructed inductively as follows.
	By non-uniform specification, there is an $x_{\xi_1}$ such that
	\begin{equation*}
		\begin{split}
			d(f^ix_{\xi_1},z_0)\leq\varepsilon_1\quad&\text{for}\quad i=a_0,\\
			d(f^ix_{\xi_1},f^{i-a_{11}}p_{\xi_1})\leq\varepsilon_1\quad&\text{for}\quad a_{11}\leq i\leq b_{11},\\
			d(f^ix_{\xi_1},z_{1,r})\leq\varepsilon_1\quad&\text{for}\quad i=q_{1r},1\leq r\leq m_1,\\
			d(f^ix_{\xi_1},f^{i-s_{11}}y_1)\leq\varepsilon_1\quad&\text{for}\quad s_{11}\leq i\leq t_{11}.
		\end{split}
	\end{equation*}
	Assume that $x_{\xi_1\cdots\xi_k}$ has been defined.
	By non-uniform specification, there exists $x_{\xi_1\cdots\xi_{k+1}}\in X$ such that
	\begin{equation}\label{eq-thm-DC1-3-proof-13}
		\begin{split}
			d(f^ix_{\xi_1\cdots\xi_k\xi_{k+1}},f^ix_{\xi_1\cdots\xi_k})\leq\varepsilon_{k+1}\quad&\text{for}\quad 0\leq i\leq t_{kk},\\
			d(f^ix_{\xi_1\cdots\xi_k\xi_{k+1}},f^{i-a_{k+1,r}}p_{\xi_r})\leq\varepsilon_{k+1}\quad&\text{for}\quad a_{k+1,r}\leq i\leq b_{k+1,r},1\leq r\leq k+1,\\
			d(f^ix_{\xi_1\cdots\xi_k\xi_{k+1}},z_{k+1,r})\leq\varepsilon_{k+1}\quad&\text{for}\quad i=q_{k+1,r},1\leq r\leq m_{k+1},\\
			d(f^ix_{\xi_1\cdots\xi_k\xi_{k+1}},f^{i-s_{k+1,r}}y_r)\leq\varepsilon_{k+1}\quad&\text{for}\quad s_{k+1,r}\leq i\leq t_{k+1,r},1\leq r\leq k+1.
		\end{split}
	\end{equation}
	Inductively, we construct a sequence $\{x_{\xi_1\cdots\xi_k}\}_{k=1}^{\infty}$, which can be easily checked to be a Cauchy sequence in $\overline{B(z_0,\varepsilon_0)}$ by \eqref{eq-thm-DC1-3-proof-13}.
	Let $x_{\xi}$ be the limit of $\{x_{\xi_1\cdots\xi_k}\}_{k=1}^{\infty}$.
	Then by \eqref{eq-thm-DC1-3-proof-13}, for any $k\in\mathbb{N}$, we have
	\begin{equation}\label{eq-thm-DC1-3-proof-14}
		d(f^ix_{\xi},f^ix_{\xi_1\cdots\xi_k})\leq\varepsilon_k,\quad\forall\,0\leq i\leq t_{kk}.
	\end{equation}
	Combining \eqref{eq-thm-DC1-3-proof-13} and \eqref{eq-thm-DC1-3-proof-14}, for any $k\in\mathbb{N}$ and $1\leq j\leq k$, we obtain that
	\begin{equation}\label{eq-thm-DC1-3-proof-15}
		\begin{split}
			d(f^ix_{\xi},f^{i-a_{kj}}p_{\xi_j})\leq\varepsilon_{k-1}\quad&\text{for}\quad a_{kj}\leq i\leq b_{kj},\\
			d(f^ix_{\xi},z_{kr})\leq\varepsilon_{k-1}\quad&\text{for}\quad i=q_{kr},1\leq r\leq m_k,\\
			d(f^ix_{\xi},f^{i-s_{kj}}y_j)\leq\varepsilon_{k-1}\quad&\text{for}\quad s_{kj}\leq i\leq t_{kj}.
		\end{split}
	\end{equation}

	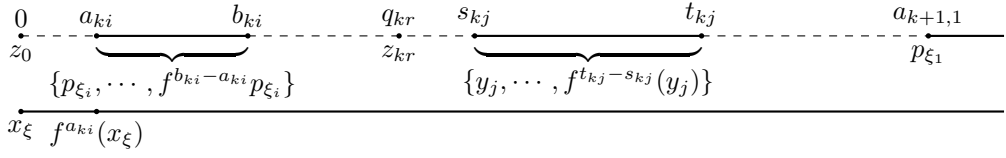
\begin{figure}[h]
		\centering
		\begin{tikzpicture}
			\draw [black, dashed] (0,2)--(1,2) (3,2)--(6,2) (9,2)--(12,2);
			\draw [thick, black] (1,2)--(3,2) (6,2)--(9,2) (12,2)--(13,2);
			\draw [thick, black] (0,1)--(13,1);
			\fill (0,2) circle (1pt);
			\fill (1,2) circle (1pt);
			\fill (3,2) circle (1pt);
			\fill (5,2) circle (1pt);
			\fill (6,2) circle (1pt);
			\fill (9,2) circle (1pt);
			\fill (12,2) circle (1pt);
			\fill (0,1) circle (1pt);
			\fill (1,1) circle (1pt);
			\draw (0,1) node [below] {$x_{\xi}$};
			\draw (1,1) node [below] {$f^{a_{ki}}(x_{\xi})$};
			\draw (0,2) node [above] {$0$};
			\draw (1,2) node [above] {$a_{ki}$};
			\draw (3,2) node [above] {$b_{ki}$};
			\draw (5,2) node [above] {$q_{kr}$};
			\draw (6,2) node [above] {$s_{kj}$};
			\draw (9,2) node [above] {$t_{kj}$};
			\draw (12,2) node [above] {$a_{k+1,1}$};
			\draw (0,2) node [below] {$z_0$};
			\draw (5,2) node [below] {$z_{kr}$};
			\draw (12,2) node [below] {$p_{\xi_1}$};
			\draw (2,2) node [below] {$\underbrace{\hspace{2cm}}$};
			\draw (7.5,2) node [below] {$\underbrace{\hspace{3cm}}$};
			\draw (2,1.7) node [below] {$\{p_{\xi_i},\cdots,f^{b_{ki}-a_{ki}}p_{\xi_i}\}$};
			\draw (7.5,1.7) node [below] {$\{y_j,\cdots,f^{t_{kj}-s_{kj}}(y_j)\}$};
		\end{tikzpicture}
		\caption{Construction of $x_{\xi}$.}
	\end{figure}

	Let $S=\{x_{\xi}:\xi\in\{0,1\}^{\mathbb{N}}\}$.
	We claim that $S$ is an uncountable DC1-scrambled set contained in
	\begin{equation*}
		G^K\cap U\cap\mathrm{Trans}(f)\setminus\bigcup_{\nu\in L} G^{\nu}.
	\end{equation*}
	Similar to the proof of Theorem~\ref{DC1-thm1}, one can show that $S$ is an uncountable DC1-scrambled set.
	We proceed to show that $S\subset G^K\cap U\cap\mathrm{Trans}(f)\setminus\bigcup_{\nu\in L} G^{\nu}$.
	Since $\overline{B(z_0,\varepsilon_0)}\subset\overline{B(z_0,3\zeta)}\subset U$ and $\{B(z_{kl},\varepsilon_k):1\leq l\leq m_k\}$ covers $X$ for all $k$, by \eqref{eq-thm-DC1-3-proof-15}, we obtain that $S\subset U\cap\mathrm{Trans}(f)$.
	It remains to show that $S\subset G^K\setminus\bigcup_{\nu\in L}G^{\nu}$.

	For any $x_{\xi}\in S$ and any $\mu\in K$, we show that $\mu\in V_f(x_{\xi})$.
	By \eqref{eq-thm-DC1-3-proof-6}, we can choose an increasing sequence $\{k(r)\}_{r=1}^{\infty}$ of positive integers such that
	\begin{equation}\label{eq-thm-DC1-3-proof-16}
		\rho(\mu_{k(r)},\mu)<2^{-r},\quad\forall r\geq1.
	\end{equation}
	Then by \eqref{eq-thm-DC1-3-proof-15}, for any $r\in\mathbb{N}$, any $j\geq k(r)$, we have
	\begin{equation*}
		d(f^ix_{\xi},f^{i-s_{j,k(r)}}y_{k(r)})\leq\varepsilon_{j-1},\quad s_{j,k(r)}\leq i\leq t_{j,k(r)}.
	\end{equation*}
	In particular, since $t_{j,k(r)}>s_{j,k(r)}+n_{k(r)}(l_{j,k(r)})$, we obtain that
	\begin{equation*}
		d(f^{s_{j,k(r)}+i}x_{\xi},f^iy_{k(r)})\leq\varepsilon_{j-1},\quad 0\leq i\leq n_{k(r)}(l_{j,k(r)})-1.
	\end{equation*}
	Then by Proposition~\ref{pre-metric-prop1} and Lemma~\ref{pre-metric-lemma2}, we obtain that
	\begin{equation}\label{eq-thm-DC1-3-proof-17}
		\begin{split}
			\rho(\delta_{x_{\xi}}^{s_{j,k(r)}+n_{k(r)}(l_{j,k(r)})},\delta_{y_{k(r)}}^{n_{k(r)}(l_{j,k(r)})})&\leq\varepsilon_{j-1}+\frac{s_{j,k(r)}}{s_{j,k(r)}+n_{k(r)}(l_{j,k(r)})}\mathrm{diam}(X)\\
			&<\varepsilon_{j-1}+\frac{\mathrm{diam}(X)}{s_{j,k(r)}}.
		\end{split}
	\end{equation}
	Combining \eqref{eq-thm-DC1-3-proof-7} and \eqref{eq-thm-DC1-3-proof-17}, we obtain that
	\begin{equation}\label{eq-thm-DC1-3-proof-18}
		\begin{split}
			\rho(\delta_{x_{\xi}}^{s_{j,k(r)}+n_{k(r)}(l_{j,k(r)})},\mu_{k(r)})&\leq\rho(\delta_{x_{\xi}}^{s_{j,k(r)}+n_{k(r)}(l_{j,k(r)})},\delta_{y_{k(r)}}^{n_{k(r)}(l_{j,k(r)})})+\rho(\delta_{y_{k(r)}}^{n_{k(r)}(l_{j,k(r)})},\mu_{k(r)})\\
			&<2^{-l_{j,k(r)}}+\varepsilon_{j-1}+\frac{\mathrm{diam}(X)}{s_{j,k(r)}}.
		\end{split}
	\end{equation}
	We can choose an increasing sequence $\{j(r)\}_{r=1}^{\infty}$ such that $s_{j(r),k(r)}>2^r\mathrm{diam}(X)$ and $l_{j(r),k(r)}>r$, then $l_{j(r),k(r)}\to\infty$ as $r\to\infty$.
	Let
	\begin{equation*}
		u(r)=s_{j(r),k(r)}+n_{k(r)}(l_{j(r),k(r)}),
	\end{equation*}
	then $\{u(r)\}_{r=1}^{\infty}$ is increasing.
	Combining \eqref{eq-thm-DC1-3-proof-16} and \eqref{eq-thm-DC1-3-proof-18}, we obtain that
	\begin{equation*}
		\rho(\delta_{x_{\xi}}^{u(r)},\mu)\leq\rho(\delta_{x_{\xi}}^{s_{j,k(r)}+n_{k(r)}(l_{j,k(r)})},\mu_{k(r)})+\rho(\mu_{k(r)},\mu)<2^{-r+1}+2^{-l_{j(r),k(r)}}+\varepsilon_{j(r)-1}.
	\end{equation*}
	Hence
	\begin{equation*}
		\lim_{r\to\infty}\rho(\delta_{x_{\xi}}^{u(r)},\mu)=0.
	\end{equation*}
	This implies that $\mu\in V_f(x_{\xi})$.
	By the arbitrariness of $\mu\in K$ and $x_{\xi}\in S$, we conclude that $S\subset G^K$.

	Now we show that $S\cap G^{\nu}=\varnothing$ for all $\nu\in L$.
	For any $x_{\xi}\in S$, we aim to show that
	\begin{equation}\label{eq-thm-DC1-3-proof-19}
		\liminf_{n\to\infty}\max_{1\leq j\leq m}\bigg|\int_X\varphi_j\mathrm{d}\delta_{x_\xi}^n-\int_X\varphi_j\mathrm{d}\nu\bigg|\geq\eta>0.
	\end{equation}
	Fix any integer $n\geq a_{2,1}$.
	We prove that there exists $1\leq r\leq m$ such that
	\begin{equation}\label{eq-thm-DC1-3-proof-20}
		\int_X\varphi_r\mathrm{d}\delta_{x_\xi}^n<\int_X\varphi_r\mathrm{d}\nu-\eta
	\end{equation}
	for the following seven cases illustrated in Figure~\ref{fig-thm-DC1-3-proof-cases}.
	\begin{figure}[h]
		\centering
		\begin{tikzpicture}
			\draw (-2,3) node {Case 1/Case 4};
			\draw [black, dashed] (0,3)--(1,3) (4,3)--(5,3) (10,3)--(11,3);
			\draw [thick, black] (1,3)--(4,3) (5,3)--(10,3);
			\fill (0,3) circle (1pt);
			\fill (1,3) circle (1pt);
			\fill (4,3) circle (1pt);
			\fill (5,3) circle (1pt);
			\fill (6,3) circle (2pt);
			\fill (8,3) circle (1pt);
			\fill (10,3) circle (1pt);
			\draw (0,3) node [below] {$0$};
			\draw (1.3,3) node [above] {\small $a_{k,j-1}/s_{k,j-1}$};
			\draw (3.5,3) node [below] {\small $b_{k,j-1}/t_{k,j-1}$};
			\draw (5,3) node [above] {\small $a_{kj}/s_{kj}$};
			\draw (6,3) node [below] {$n$};
			\draw (8,3) node [above] {\small $a_{kj}+N/s_{kj}+N_j$};
			\draw (10,3) node [below] {\small $b_{kj}/t_{kj}$};
			\draw (-2,2) node {Case 2/Case 5};
			\draw [black, dashed] (0,2)--(1,2) (4,2)--(5,2) (10,2)--(11,2);
			\draw [thick, black] (1,2)--(4,2) (5,2)--(10,2);
			\fill (0,2) circle (1pt);
			\fill (1,2) circle (1pt);
			\fill (4,2) circle (1pt);
			\fill (5,2) circle (1pt);
			\fill (7,2) circle (1pt);
			\fill (8.5,2) circle (2pt);
			\fill (10,2) circle (1pt);
			\draw (0,2) node [below] {$0$};
			\draw (1.3,2) node [above] {\small $a_{k,j-1}/s_{k,j-1}$};
			\draw (3.5,2) node [below] {\small $b_{k,j-1}/t_{k,j-1}$};
			\draw (5,2) node [above] {\small $a_{kj}/s_{kj}$};
			\draw (7,2) node [below] {\small $a_{kj}+N/s_{kj}+N_j$};
			\draw (8.5,2) node [above] {$n$};
			\draw (10,2) node [below] {\small $b_{kj}/t_{kj}$};
			\draw (-2,1) node {Case 3/Case 6};
			\draw [black, dashed] (0,1)--(1,1) (4,1)--(7,1);
			\draw [thick, black] (1,1)--(4,1) (7,1)--(11,1);
			\fill (0,1) circle (1pt);
			\fill (1,1) circle (1pt);
			\fill (4,1) circle (1pt);
			\fill (5,1) circle (2pt);
			\fill (7,1) circle (1pt);
			\draw (0,1) node [below] {$0$};
			\draw (1,1) node [above] {\small $a_{kj}/s_{kj}$};
			\draw (4,1) node [below] {\small $b_{kj}/t_{kj}$};
			\draw (5,1) node [above] {$n$};
			\draw (7,1) node [below] {\small $a_{k,j+1}/s_{k,j+1}$};
			\draw (-2,0) node {Case 7};
			\draw [black, dashed] (0,0)--(1,0) (3,0)--(5,0) (10,0)--(11,0);
			\draw [thick, black] (1,0)--(3,0) (5,0)--(10,0);
			\fill (0,0) circle (1pt);
			\fill (1,0) circle (1pt);
			\fill (3,0) circle (1pt);
			\fill (5,0) circle (1pt);
			\fill (7,0) circle (1pt);
			\fill (9,0) circle (2pt);
			\fill (10,0) circle (1pt);
			\draw (0,0) node [below] {$0$};
			\draw (1,0) node [above] {\small $a_{kk}/s_{kk}$};
			\draw (3,0) node [below] {\small $b_{kk}/t_{kk}$};
			\draw (5,0) node [above] {\small $s_{k1}/a_{k+1,1}$};
			\draw (7,0) node [below] {\small $s_{k1}+N_1/a_{k+1,1}+N$};
			\draw (9,0) node [above] {$n$};
			\draw (10,0) node [below] {\small $t_{k1}/b_{k+1,1}$};
		\end{tikzpicture}
		\caption{The seven cases for the location of $n$.}
		\label{fig-thm-DC1-3-proof-cases}
	\end{figure}
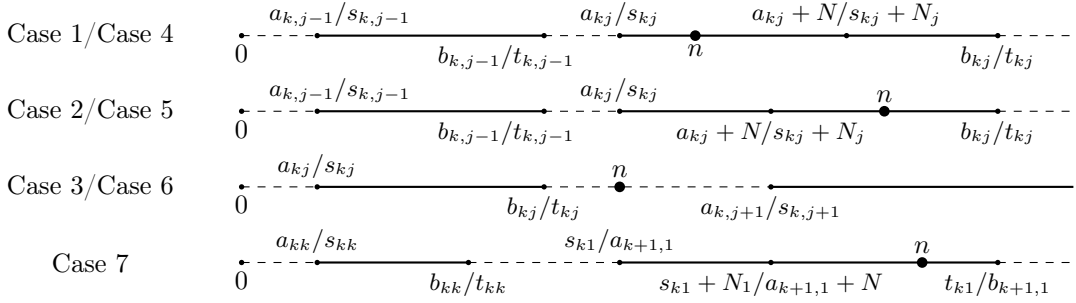

	{\bf Case 1.}
	Assume that there exists $j$ with $2\leq j\leq k+1$ such that $a_{kj}\leq n\leq a_{kj}+N$, where $N$ is replaced by $N_1$ if $j=k+1$ (recall that $a_{k,k+1}=s_{k1}$).
	Let $\widehat{\mu}=\delta_{p_{\xi_{j-1}}}^{b_{k,j-1}-a_{k,j-1}+1}$.
	Then by \eqref{eq-thm-DC1-3-proof-8}, we have $\widehat{\mu}\in\mathcal{B}(\widetilde{K},\kappa/3)$.
	By definition of $\kappa$, we can choose $1\leq r\leq m$ such that $(\widehat{\mu},\nu)\in\mathcal{U}_r\times\mathcal{V}_r$.
	Combining \eqref{eq-thm-DC1-3-proof-1}, \eqref{eq-thm-DC1-3-proof-5}, \eqref{eq-thm-DC1-3-proof-12} and \eqref{eq-thm-DC1-3-proof-15}, we obtain that
	\begin{equation*}
		\begin{split}
			\int_X\varphi_r\mathrm{d}\delta_{x_{\xi}}^n&\leq\frac{1}{n}\sum_{i=a_{k,j-1}}^{b_{k,j-1}}\varphi_r(f^ix_{\xi})+\frac{n-b_{k,j-1}+a_{k,j-1}-1}{n}\\
			&\leq\frac{1}{n}\sum_{i=0}^{b_{k,j-1}-a_{k,j-1}}\varphi_r(f^ip_{\xi_{j-1}})+\eta+\frac{n-b_{k,j-1}+a_{k,j-1}-1}{n}\\
			&=\frac{b_{k,j-1}-a_{k,j-1}+1}{n}\int_X\varphi_r\mathrm{d}\widehat{\mu}+\eta+\frac{n-b_{k,j-1}+a_{k,j-1}-1}{n}\\
			&<\frac{1+\theta}{1+\tau-2\theta}\int_X\varphi_r\mathrm{d}\widehat{\mu}+\eta+\frac{\tau+4\theta}{1+\tau-2\theta}<\int_X\varphi_r\mathrm{d}\nu-\eta.
		\end{split}
	\end{equation*}

	{\bf Case 2.}
	Assume that there exists $j$ with $2\leq j\leq k$ such that $a_{kj}+N+1\leq n\leq b_{kj}$.
	Let
	\begin{equation*}
		\widehat{\mu}=\frac{b_{k,j-1}-a_{k,j-1}+1}{n-a_{kj}+b_{k,j-1}-a_{k,j-1}+1}\delta_{p_{\xi_{j-1}}}^{b_{k,j-1}-a_{k,j-1}+1}+\frac{n-a_{kj}}{n-a_{kj}+b_{k,j-1}-a_{k,j-1}+1}\delta_{p_{\xi_j}}^{n-a_{kj}}.
	\end{equation*}
	Then by \eqref{eq-thm-DC1-3-proof-8}, we have $\widehat{\mu}\in\mathcal{B}(\widetilde{K},2\kappa/3)$.
	By  definition of $\kappa$, we can choose $1\leq r\leq m$ such that $(\widehat{\mu},\nu)\in\mathcal{U}_r\times\mathcal{V}_r$.
	Let $\widehat{\mathbb{I}}=\mathbb{N}\cap([a_{k,j-1},b_{k,j-1}]\cup[a_{kj},n-1])$.
	Note that
	\begin{equation*}
		\frac{n-a_{kj}+b_{k,j-1}-a_{k,j-1}+1}{n}\geq\frac{b_{k,j-1}-a_{k,j-1}+1}{a_{kj}}\quad\text{and}\quad 0\leq\int_X\varphi_r\mathrm{d}\widehat{\mu}\leq1.
	\end{equation*}
	Combining \eqref{eq-thm-DC1-3-proof-1}, \eqref{eq-thm-DC1-3-proof-5}, \eqref{eq-thm-DC1-3-proof-12} and \eqref{eq-thm-DC1-3-proof-15}, we obtain that
	\begin{equation*}
		\begin{split}
			\int_X\varphi_r\mathrm{d}\delta_{x_{\xi}}^n&\leq\frac{1}{n}\sum_{i\in\widehat{\mathbb{I}}}\varphi_r(f^ix_{\xi})+\frac{a_{kj}-b_{k,j-1}+a_{k,j-1}-1}{n}\\
			&\leq\frac{n-a_{kj}+b_{k,j-1}-a_{k,j-1}+1}{n}\int_X\varphi_r\mathrm{d}\widehat{\mu}+\eta+\frac{a_{kj}-b_{k,j-1}+a_{k,j-1}-1}{n}\\
			&\leq\frac{b_{k,j-1}-a_{k,j-1}+1}{a_{kj}}\int_X\varphi_r\mathrm{d}\widehat{\mu}+\eta+\frac{a_{kj}-b_{k,j-1}+a_{k,j-1}-1}{a_{kj}}\\
			&\leq\frac{1+\theta}{1+\tau-2\theta}\int_X\varphi_r\mathrm{d}\widehat{\mu}+\eta+\frac{\tau+4\theta}{1+\tau-2\theta}<\int_X\varphi_r\mathrm{d}\nu-\eta.
		\end{split}
	\end{equation*}

	{\bf Case 3.}
	Assume that there exists $j$ with $1\leq j\leq k$ such that $b_{kj}+1\leq n\leq a_{k,j+1}-1$ (recall that $a_{k,k+1}=s_{k1}$).
	Let $\widehat{\mu}=\delta_{p_{\xi_j}}^{b_{kj}-a_{kj}+1}$.
	By \eqref{eq-thm-DC1-3-proof-8}, we have $\widehat{\mu}\in\mathcal{B}(\widetilde{K},\kappa/3)$.
	By definition of $\kappa$, we can choose $1\leq r\leq m$ such that $(\widehat{\mu},\nu)\in\mathcal{U}_r\times\mathcal{V}_r$.
	Combining \eqref{eq-thm-DC1-3-proof-1}, \eqref{eq-thm-DC1-3-proof-5}, \eqref{eq-thm-DC1-3-proof-12} and \eqref{eq-thm-DC1-3-proof-15}, we obtain that
	\begin{equation*}
		\begin{split}
			\int_X\varphi_r\mathrm{d}\delta_{x_{\xi}}^n&\leq\frac{1}{n}\sum_{i=a_{kj}}^{b_{kj}}\varphi_r(f^ix_{\xi})+\frac{n-b_{kj}+a_{kj}-1}{n}\\
			&\leq\frac{1}{n}\sum_{i=0}^{b_{kj}-a_{kj}}\varphi_r(f^ip_{\xi_j})+\eta+\frac{n-b_{kj}+a_{kj}-1}{n}\\
			&=\frac{b_{kj}-a_{kj}+1}{n}\int_X\varphi_r\mathrm{d}\widehat{\mu}+\eta+\frac{n-b_{kj}+a_{kj}-1}{n}\\
			&\leq\frac{1+\theta}{1+\tau-2\theta}\int_X\varphi_r\mathrm{d}\widehat{\mu}+\eta+\frac{\tau+4\theta}{1+\tau-2\theta}<\int_X\varphi_r\mathrm{d}\nu-\eta.
		\end{split}
	\end{equation*}

	{\bf Case 4.}
	Assume that there exists $j$ with $2\leq j\leq k+1$ such that $s_{kj}\leq n\leq s_{kj}+N_j$, where $N_j$ is replaced by $N$ if $j=k+1$ (recall that $s_{k,k+1}=a_{k+1,1}$).
	The proof for this case is similar to Case~1.
	Let $\widehat{\mu}=\delta_{y_{j-1}}^{t_{k,j-1}-s_{k,j-1}+1}$.
	Then by \eqref{eq-thm-DC1-3-proof-9}, we have $\widehat{\mu}\in\mathcal{B}(\widetilde{K},\kappa/3)$.
	By definition of $\kappa$, we can choose $1\leq r\leq m$ such that $(\widehat{\mu},\nu)\in\mathcal{U}_r\times\mathcal{V}_r$.
	Combining \eqref{eq-thm-DC1-3-proof-1}, \eqref{eq-thm-DC1-3-proof-5}, \eqref{eq-thm-DC1-3-proof-12} and \eqref{eq-thm-DC1-3-proof-15}, we obtain that
	\begin{equation*}
		\begin{split}
			\int_X\varphi_r\mathrm{d}\delta_{x_{\xi}}^n&\leq\frac{1}{n}\sum_{i=s_{k,j-1}}^{t_{k,j-1}}\varphi_r(f^ix_{\xi})+\frac{n-t_{k,j-1}+s_{k,j-1}-1}{n}\\
			&\leq\frac{1}{n}\sum_{i=0}^{t_{k,j-1}-s_{k,j-1}}\varphi_r(f^iy_{j-1})+\eta+\frac{n-t_{k,j-1}+s_{k,j-1}-1}{n}\\
			&=\frac{t_{k,j-1}-s_{k,j-1}+1}{n}\int_X\varphi_r\mathrm{d}\widehat{\mu}+\eta+\frac{n-t_{k,j-1}+s_{k,j-1}-1}{n}\\
			&<\frac{1+\theta}{1+\tau-2\theta}\int_X\varphi_r\mathrm{d}\widehat{\mu}+\eta+\frac{\tau+4\theta}{1+\tau-2\theta}<\int_X\varphi_r\mathrm{d}\nu-\eta.
		\end{split}
	\end{equation*}

	{\bf Case 5.}
	Assume that there exists $j$ with $2\leq j\leq k$ such that $s_{kj}+N_j+1\leq n\leq t_{kj}$.
	The proof for this case is similar to Case~2.
	Let
	\begin{equation*}
		\widehat{\mu}=\frac{t_{k,j-1}-s_{k,j-1}+1}{n-s_{kj}+t_{k,j-1}-s_{k,j-1}+1}\delta_{y_{j-1}}^{t_{k,j-1}-s_{k,j-1}+1}+\frac{n-s_{kj}}{n-s_{kj}+t_{k,j-1}-s_{k,j-1}+1}\delta_{y_j}^{n-s_{kj}}.
	\end{equation*}
	Then by \eqref{eq-thm-DC1-3-proof-9}, we have $\widehat{\mu}\in\mathcal{B}(\widetilde{K},2\kappa/3)$.
	By definition of $\kappa$, we can choose $1\leq r\leq m$ such that $(\widehat{\mu},\nu)\in\mathcal{U}_r\times\mathcal{V}_r$.
	Let $\widehat{\mathbb{I}}=\mathbb{N}\cap([s_{k,j-1},t_{k,j-1}]\cup[s_{kj},n-1])$.
	Note that
	\begin{equation*}
		\frac{n-s_{kj}+t_{k,j-1}-s_{k,j-1}+1}{n}\geq\frac{t_{k,j-1}-s_{k,j-1}+1}{s_{kj}}\quad\text{and}\quad 0\leq\int_X\varphi_r\mathrm{d}\widehat{\mu}\leq1.
	\end{equation*}
	Combining \eqref{eq-thm-DC1-3-proof-1}, \eqref{eq-thm-DC1-3-proof-5}, \eqref{eq-thm-DC1-3-proof-12} and \eqref{eq-thm-DC1-3-proof-15}, we obtain that
	\begin{equation*}
		\begin{split}
			\int_X\varphi_r\mathrm{d}\delta_{x_{\xi}}^n&\leq\frac{1}{n}\sum_{i\in\widehat{\mathbb{I}}}\varphi_r(f^ix_{\xi})+\frac{s_{kj}-t_{k,j-1}+s_{k,j-1}-1}{n}\\
			&\leq\frac{n-s_{kj}+t_{k,j-1}-s_{k,j-1}+1}{n}\int_X\varphi_r\mathrm{d}\widehat{\mu}+\eta+\frac{s_{kj}-t_{k,j-1}+s_{k,j-1}-1}{n}\\
			&\leq\frac{t_{k,j-1}-s_{k,j-1}+1}{s_{kj}}\int_X\varphi_r\mathrm{d}\widehat{\mu}+\eta+\frac{s_{kj}-t_{k,j-1}+s_{k,j-1}-1}{s_{kj}}\\
			&<\frac{1+\theta}{1+\tau-2\theta}\int_X\varphi_r\mathrm{d}\widehat{\mu}+\eta+\frac{\tau+4\theta}{1+\tau-2\theta}<\int_X\varphi_r\mathrm{d}\nu-\eta.
		\end{split}
	\end{equation*}

	{\bf Case 6.}
	Assume that there exists $j$ with $1\leq j\leq k$ such that $t_{kj}+1\leq n\leq s_{k,j+1}-1$ (recall that $s_{k,k+1}=a_{k+1,1}$).
	The proof for this case is similar to Case~3.
	Let $\widehat{\mu}=\delta_{y_j}^{t_{kj}-s_{kj}+1}$.
	Then by \eqref{eq-thm-DC1-3-proof-9}, we have $\widehat{\mu}\in\mathcal{B}(\widetilde{K},\kappa/3)$.
	By definition of $\kappa$, we can choose $1\leq r\leq m$ such that $(\widehat{\mu},\nu)\in\mathcal{U}_r\times\mathcal{V}_r$.
	Combining \eqref{eq-thm-DC1-3-proof-1}, \eqref{eq-thm-DC1-3-proof-5}, \eqref{eq-thm-DC1-3-proof-12} and \eqref{eq-thm-DC1-3-proof-15}, we obtain that
	\begin{equation*}
		\begin{split}
			\int_X\varphi_r\mathrm{d}\delta_{x_{\xi}}^n&\leq\frac{1}{n}\sum_{i=s_{kj}}^{t_{kj}}\varphi_r(f^ix_{\xi})+\frac{n-t_{kj}+s_{kj}-1}{n}\\
			&\leq\frac{1}{n}\sum_{i=0}^{t_{kj}-s_{kj}}\varphi_r(f^iy_j)+\eta+\frac{n-t_{kj}+s_{kj}-1}{n}\\
			&=\frac{t_{kj}-s_{kj}+1}{n}\int_X\varphi_r\mathrm{d}\widehat{\mu}+\eta+\frac{n-t_{kj}+s_{kj}-1}{n}\\
			&\leq\frac{1+\theta}{1+\tau-2\theta}\int_X\varphi_r\mathrm{d}\widehat{\mu}+\eta+\frac{\tau+4\theta}{1+\tau-2\theta}<\int_X\varphi_r\mathrm{d}\nu-\eta.
		\end{split}
	\end{equation*}

	{\bf Case 7.}
	Assume that $a_{k1}+N+1\leq n\leq b_{k1}$ or $s_{k1}+N_1+1\leq n\leq t_{k1}$.
	The proof of this case is similar to Case~2 and Case~5.
	We only consider the case $s_{k1}+N_1+1\leq n\leq t_{k1}$, while the other is similar.
	Let
	\begin{equation*}
		\widehat{\mu}=\frac{b_{kk}-a_{kk}+1}{n-s_{k1}+b_{kk}-a_{kk}+1}\delta_{p_{\xi_k}}^{b_{kk}-a_{kk}+1}+\frac{n-s_{k1}}{n-s_{k1}+b_{kk}-a_{kk}+1}\delta_{y_1}^{n-s_{k1}}.
	\end{equation*}
	Then by \eqref{eq-thm-DC1-3-proof-9} and \eqref{eq-thm-DC1-3-proof-8}, we have $\widehat{\mu}\in\mathcal{B}(\widetilde{K},2\kappa/3)$.
	By definition of $\kappa$, we can choose $1\leq r\leq m$ such that $(\widehat{\mu},\nu)\in\mathcal{U}_r\times\mathcal{V}_r$.
	Let $\widehat{\mathbb{I}}=\mathbb{N}\cap([a_{kk},b_{kk}]\cup[s_{k1},n-1])$.
	\begin{equation*}
		\frac{n-s_{k1}+b_{kk}-a_{kk}+1}{n}\geq\frac{b_{kk}-a_{kk}+1}{s_{k1}}\quad\text{and}\quad 0\leq\int_X\varphi_r\mathrm{d}\widehat{\mu}\leq1.
	\end{equation*}
	Combining \eqref{eq-thm-DC1-3-proof-1}, \eqref{eq-thm-DC1-3-proof-5}, \eqref{eq-thm-DC1-3-proof-12} and \eqref{eq-thm-DC1-3-proof-15}, we obtain that
	\begin{equation*}
		\begin{split}
			\int_X\varphi_r\mathrm{d}\delta_{x_{\xi}}^n&\leq\frac{1}{n}\sum_{i\in\widehat{\mathbb{I}}}\varphi_r(f^ix_{\xi})+\frac{s_{k1}-b_{kk}+a_{kk}-1}{n}\\
			&\leq\frac{n-s_{k1}+b_{kk}-a_{kk}+1}{n}\int_X\varphi_r\mathrm{d}\widehat{\mu}+\eta+\frac{s_{k1}-b_{kk}+a_{kk}-1}{n}\\
			&\leq\frac{b_{kk}-a_{kk}+1}{s_{k1}}\int_X\varphi_r\mathrm{d}\widehat{\mu}+\eta+\frac{s_{k1}-b_{kk}+a_{kk}-1}{s_{k1}}\\
			&<\frac{1+\theta}{1+\tau-2\theta}\int_X\varphi_r\mathrm{d}\widehat{\mu}+\eta+\frac{\tau+4\theta}{1+\tau-2\theta}<\int_X\varphi_r\mathrm{d}\nu-\eta.
		\end{split}
	\end{equation*}

	It follows from the above seven cases that \eqref{eq-thm-DC1-3-proof-20} holds for all $n\geq a_{21}$.
	Therefore, \eqref{eq-thm-DC1-3-proof-19} holds, which implies that $\nu\notin V_f(x_{\xi})$.
	By the arbitrariness of $\nu\in L$ and $x_{\xi}\in S$, we conclude that
	\begin{equation*}
		S\cap\Big(\bigcup_{\nu\in L}G^{\nu}\Big)=\varnothing.
	\end{equation*}
	This completes the proof of Theorem~\ref{thm-DC1-3}.
\end{proof}

Theorem~\ref{thm-DC1-2} follows immediately from the following result.

\begin{theorem}\label{DC1-thm2}
	Suppose that $(X,f)$ satisfies the non-uniform $\tau$-specification property.
	If $\tau<\infty$, then for any $m+n$ distinct ergodic measures $\mu_1,\cdots,\mu_m,\nu_1,\cdots,\nu_n\in\mathcal{M}_f^{erg}(X)$ and any non-empty open set $U\subset X$, there exists an uncountable DC1-scrambled set contained in
	\begin{equation*}
		\Big(\bigcap_{1\leq i\leq m}\bigcap_{1\leq j\leq n}(G^{\mu_i}\setminus G^{\nu_j})\Big)\cap U\cap\mathrm{Trans}(f).
	\end{equation*}
\end{theorem}

\begin{proof}
	Let $M(n,\varepsilon)$ be an admissible gap function with $\tau(M)=\tau$.
	Let $K=\{\mu_1,\cdots,\mu_m\}$ and $L=\{\nu_1,\cdots,\nu_n\}$.
	By Theorem~\ref{distal-thm5}, one can choose an ergodic measure
	\begin{equation}\label{eq-DC1-thm2-proof-1}
		\lambda\in\mathcal{M}_f^{erg}(X)\setminus\{\mu_1,\cdots,\mu_m,\nu_1,\cdots,\nu_n\}
	\end{equation}
	such that $G_{\lambda}$ contains a distal pair $\{p_0,p_1\}$, then $V_f(p_0)=V_f(p_1)=\{\lambda\}$.
	
	It suffices to show that $K$, $L$ and $\{p_0,p_1\}$ satisfy conditions~(1) and~(2) of Theorem~\ref{thm-DC1-3}.
	It is known that every ergodic measure has generic points.
	Hence $K=\{\mu_1,\cdots,\mu_m\}\subset\mathcal{M}_f^{erg}(X)$ satisfies condition~(1) of Theorem~\ref{thm-DC1-3}.
	It remains to show that condition~(2) of Theorem~\ref{thm-DC1-3} holds. 
	Note that any two distinct ergodic measures $\mu$ and $\nu$ satisfy $\rho_{\mathsf{tv}}(\mu,\nu)=1$.
	By Proposition~\ref{pre-tv-prop1}, the set
	\begin{equation*}
		\widetilde{K}=\overline{\mathrm{co}}(\{\lambda\}\cup K)=\mathrm{co}\{\mu_1,\cdots,\mu_m,\lambda\}
	\end{equation*}
	satisfies
	\begin{equation*}
		\rho_{\mathsf{tv}}(\mu,\nu)=1,\quad\forall\,(\mu,\nu)\in\widetilde{K}\times L.
	\end{equation*}
	Hence the condition~(2) of Theorem~\ref{thm-DC1-3} holds.
	This completes the proof.
\end{proof}

\begin{proof}[{\bf Proof of Theorem~\ref{thm-DC1-2}}]
	Theorem~\ref{thm-DC1-2} follows directly from Theorem~\ref{DC1-thm2} by taking $m=n=1$, $\mu_1=\mu$ and $\nu_1=\nu$.
\end{proof}

\section{Examples}\label{Sect-example}

In this section, we will prove Theorem~\ref{thm-example-1} and Theorem~\ref{thm-example-2}, and provide an example with non-uniform specification which satisfies conditions of Theorem~\ref{thm-DC1-3}.
The following result provides a method for constructing subshifts with the non-uniform specification property.

\begin{theorem}[{\cite[Theorem~B]{Lin-Tian-Yu-2025+}}]\label{example-thm1}
	Given $\tau\in(0,\infty]$, a finite alphabet $\mathcal{A}$, and a subshift $Z\subset\mathcal{A}^{\mathbb{N}_0}$, there exists a one-sided subshift $(X,f)$ satisfying the non-uniform periodic $\tau$-specification property with gap function $M(n,\varepsilon)$ such that
	\begin{enumerate}
		\item for $\varepsilon>0$ sufficiently small, we have
		\begin{equation*}
			\lim_{n\rightarrow\infty}\frac{M(n,\varepsilon)}{n}=\tau;
		\end{equation*}
		\item $Z\subset X$;
		\item $\mathcal{M}_f^{erg}(X)$ can be written as $\mathcal{M}_f^{erg}(X)=\mathcal{M}_0\cup\mathcal{M}_f^{erg}(Z)$ such that $\overline{\mathcal{M}_0}\cap\overline{\mathcal{M}_f^{erg}(Z)}=\varnothing$;
		\item for any $x\in X\setminus\bigcup_{i\geq 0}f^{-i}Z$, we have $V_f(x)\cap\overline{\mathcal{M}_0}\neq\varnothing$;
		\item if $0<\tau<\infty$, then $(X,f)$ can be constructed such that for any $\mu\in\overline{\mathcal{M}_0}$, we have
		\begin{equation*}
			h_{\mu}(f)\leq \log 2+\frac{1}{1+\tau}\log|\mathcal{A}|;
		\end{equation*}
		if $\tau=\infty$, then for any $\delta>0$, $(X,f)$ can be constructed such that for any $\mu\in\overline{\mathcal{M}_0}$, we have
		\begin{equation*}
			h_{\mu}(f)\leq \log 2+\delta\log|\mathcal{A}|;
		\end{equation*}
	\end{enumerate}
\end{theorem}

\begin{remark}\label{example-rmk2}
	In the proof of Theorem~\ref{example-thm1}(3) in \cite[\S~4.1]{Lin-Tian-Yu-2025+}, it was proved that $\overline{\mathcal{M}_0}\cap\mathcal{M}_f(Z)=\varnothing$.
	See \cite[\S~4.1, (4.1) and (4.2)]{Lin-Tian-Yu-2025+}.
	Thus for any $x\in X\setminus\bigcup_{i\geq 0}f^{-i}Z$, we have $V_f(x)\cap\mathcal{M}_f(X)\setminus\mathcal{M}_f(Z)\neq\varnothing$.
\end{remark}

For a finite alphabet $\mathcal{A}$, assume that the one-sided full shift $\mathcal{A}^{\mathbb{N}_0}$ is equipped with the metric
\begin{equation*}
    d(x,y)=\sum_{j=0}^{\infty}2^{-j}|x_j-y_j|.
\end{equation*}

\begin{proof}[{\bf Proof of Theorem~\ref{thm-example-1}}]
	Consider the one-sided full shift $\{0,1\}^{\mathbb{N}_0}$.
	Let $f$ denote the shift action.
	Consider the set of forbidden words
	\begin{equation*}
		\mathcal{F}=\{1^s0^t1:s,t\in\mathbb{N},\ t<s^2\}.
	\end{equation*}
	Define
	\begin{equation*}
		X:=X(\mathcal{F}):=\{x\in\{0,1\}^{\mathbb{N}_0}:x_ix_{i+1}\cdots x_l\notin\mathcal{F},\ \forall\,0\leq i\leq l\}.
	\end{equation*}
	The above construction coincides with the construction in the proof of Theorem~\ref{example-thm1} for $\tau=\infty$, with $\mathcal{A}=\{1\}$, $Z=\{1^{\infty}\}$ and $\sigma=1$.
	It follows from the proof of Theorem~\ref{example-thm1} (see \cite[\S~4.1]{Lin-Tian-Yu-2025+}) that $(X,f)$ satisfies the non-uniform $\infty$-specification property.

	Let $\mu=\delta_{1^{\infty}}$ and $\nu=\delta_{0^{\infty}}$.
	Then $\mu,\nu\in\mathcal{M}_f^{erg}(X)$ and $\mu\neq\nu$.
	We show that
	\begin{equation}\label{eq-thm-example-1-proof-1}
		G^{\mu}\setminus G^{\nu}=\bigcup_{j\geq0}f^{-j}\{1^{\infty}\}.
	\end{equation}
	Each point in the right-hand side is generic for $\mu$, and therefore belongs to $G^{\mu}\setminus G^{\nu}$.
	Fix
	\begin{equation*}
		x\in G^{\mu}\setminus\bigcup_{j\geq0}f^{-j}\{1^{\infty}\}.
	\end{equation*}
	Then $x$ has infinitely many $0$'s and $1$'s.
	Thus $x$ can be written as
	\begin{equation*}
		x=0^{t_0}1^{s_1}0^{t_1}1^{s_2}0^{t_2}\cdots,\quad t_0\in\mathbb{N}_0,\; s_i,t_i\in\mathbb{N},\; t_i\geq s_i^2.
	\end{equation*}
	Let
	\begin{equation*}
		A_i:=t_0+\sum_{j=1}^{i-1}(s_j+t_j),\quad i\geq1.
	\end{equation*}
	Choose $n_k\to\infty$ such that $\delta_x^{n_k}\to\mu$.
	Let
	\begin{equation*}
		\varepsilon_k:=\frac{1}{n_k}|\{0\leq j<n_k:x_j=0\}|.
	\end{equation*}
	Since $\delta_x^{n_k}\to\mu$, we have $\varepsilon_k\to0$.
	For all sufficiently large $k$, choose $i_k\in\mathbb{N}$ with $A_{i_k}<n_k\leq A_{i_k+1}$.
	Write $a_k=A_{i_k}$, $u_k=s_{i_k}$ and $v_k=t_{i_k}$.
	Since $t_i\geq s_i$, the first $a_k$ symbols contain at least $a_k/2$ $0$'s and at most $a_k/2$ $1$'s.
	The number of $1$'s among the first $n_k$ symbols is at most $a_k/2+u_k$.
	Hence
	\begin{equation*}
		a_k\leq2\varepsilon_kn_k,\quad\text{and}\quad u_k\geq(1-\varepsilon_k)n_k-\frac{a_k}{2}\geq(1-2\varepsilon_k)n_k.
	\end{equation*}
	It follows that $u_k\to\infty$ and $a_k/u_k\to0$.
	Since $v_k\geq u_k^2$, we have
	\begin{equation*}
		\lim_{k\to\infty}\frac{a_k+u_k}{v_k}=0.
	\end{equation*}

	For any $c\in\{0,1\}$, $a\in\mathbb{N}_0$ and $l\in\mathbb{N}$ with $x_ax_{a+1}\cdots x_{a+l-1}=c^l$, we have
	\begin{equation}\label{eq-thm-example-1-proof-2}
		\sum_{j=0}^{l-1}d(f^{a+j}x,c^{\infty})\leq\sum_{j=0}^{l-1}\sum_{h=l-j}^{\infty}2^{-h}=2(1-2^{-l})<2.
	\end{equation}
	Let $N_k=a_k+u_k+v_k$.
	Since $\mathrm{diam}(X)=2$, by \eqref{eq-thm-example-1-proof-2}, we have
	\begin{equation*}
		\begin{split}
			\rho(\delta_x^{N_k},\nu)&\leq\frac{1}{N_k}\sum_{j=0}^{N_k-1}d(f^jx,0^{\infty})\\
			&=\frac{1}{N_k}\sum_{j=0}^{a_k+u_k-1}d(f^jx,0^{\infty})+\frac{1}{N_k}\sum_{j=a_k+u_k}^{N_k-1}d(f^jx,0^{\infty})\leq\frac{2(a_k+u_k)+2}{a_k+u_k+v_k}\to0.
		\end{split}
	\end{equation*}
	Consequently, $\nu\in V_f(x)$, which proves \eqref{eq-thm-example-1-proof-1}.

	By \eqref{eq-thm-example-1-proof-1}, for any $x,y\in G^{\mu}\setminus G^{\nu}$, there exists $j\geq0$ such that $f^jx=f^jy=1^{\infty}$.
	In particular, $d(f^nx,f^ny)=0$ for every $n\geq j$, so $\{x,y\}$ is not a DC1-scrambled pair.
	This completes the proof.
\end{proof}

\begin{proof}[{\bf Proof of Theorem~\ref{thm-example-2}}]
	Consider the one-sided full shift $\{0,1\}^{\mathbb{N}_0}$.
	Let $f$ denote the shift action.
	Consider the set of forbidden words
	\begin{equation*}
		\mathcal{F}_{\tau}:=\{1^s0^t1:s,t\in\mathbb{N},\ t<\tau s\},
	\end{equation*}
	and define
	\begin{equation*}
		X:=X(\mathcal{F}_{\tau}):=\{x\in\{0,1\}^{\mathbb{N}_0}:x_ix_{i+1}\cdots x_l\notin\mathcal{F}_{\tau},\ \forall\,0\leq i\leq l\}.
	\end{equation*}
	This is the construction in the proof of \cite[Theorem~B]{Lin-Tian-Yu-2025+}, with $\mathcal{A}=\{1\}$ and $Z=\{1^{\infty}\}$.
	It follows from the proof Theorem~\ref{example-thm1} (see \cite[\S~4.1]{Lin-Tian-Yu-2025+}) that $(X,f)$ satisfies the non-uniform $\tau$-specification.

	Let $\mu:=\delta_{1^{\infty}}$ and $\omega:=\delta_{0^{\infty}}$ and
	\begin{equation*}
		\nu:=\frac{1}{1+\tau}\mu+\frac{\tau}{1+\tau}\omega.
	\end{equation*}
	Then $\rho_{\mathsf{tv}}(\mu,\nu)=\tau/(1+\tau)$.
	Let $K=\{\mu\}$ and $L=\{\nu\}$.
	The sets $K,L$ are disjoint and compact, and $K$ is connected with $1^{\infty}\in G_K$.
	Thus condition~(1) of Theorem~\ref{thm-DC1-3} holds.
	Let $m=\lceil\tau\rceil$ and take
	\begin{equation*}
		p_0:=(10^m)^{\infty}, p_1:=f(p_0)\in X \quad\text{and}\quad \lambda:=\frac{1}{m+1}\sum_{j=0}^{m}\delta_{f^jp_0}.
	\end{equation*}
	They are distinct points on the same periodic orbit, and hence form a distal pair.
	This shows (1).
	Note that $V_f(p_0)=V_f(p_1)=\{\lambda\}$.
	Hence
	\begin{equation*}
		\overline{\mathrm{co}}(K\cup V_f(p_0)\cup V_f(p_1))=\mathrm{co}\{\mu,\lambda\}.
	\end{equation*}
	Note that every $\eta\in\mathrm{co}\{\mu,\lambda\}$ satisfies $\eta(\{0^{\infty}\})=0$ and $\nu(\{0^{\infty}\})=\tau/(1+\tau)$.
	Therefore
	\begin{equation*}
		\rho_{\mathsf{tv}}(\eta,\nu)\geq\frac{\tau}{1+\tau},\quad\forall\,\eta\in\mathrm{co}\{\mu,\lambda\}.
	\end{equation*}
	Since $\rho_{\mathsf{tv}}(\mu,\nu)=\tau/(1+\tau)$, this shows (2).

	It remains to show that
	\begin{equation}\label{eq-thm-example-2-proof-1}
		G^K\setminus\bigcup_{\xi\in L}G^{\xi}=G^{\mu}\setminus G^{\nu}=\bigcup_{j\geq0}f^{-j}\{1^{\infty}\}.
	\end{equation}
	Every point in the right-hand side is generic for $\mu$ and belongs to $G^{\mu}\setminus G^{\nu}$.
	Take
	\begin{equation*}
		x\in G^{\mu}\setminus\bigcup_{j\geq0}f^{-j}\{1^{\infty}\}.
	\end{equation*}
	As in the proof of Theorem~\ref{thm-example-1}, write
	\begin{equation*}
		x=0^{t_0}1^{s_1}0^{t_1}1^{s_2}0^{t_2}\cdots,\quad\text{where}\; t_0\in\mathbb{N}_0,\; s_i,t_i\in\mathbb{N},\; t_i\geq\lceil\tau s_i\rceil,
	\end{equation*}
	and let $A_i=t_0+\sum_{j=1}^{i-1}(s_j+t_j)$.
	Choose $n_k\to\infty$ such that $\delta_x^{n_k}\to\mu$, and let
	\begin{equation*}
		\varepsilon_k:=\frac{1}{n_k}|\{0\leq j<n_k:x_j=0\}|\to0.
	\end{equation*}
	For all sufficiently large $k$, choose $i_k$ with $A_{i_k}<n_k\leq A_{i_k+1}$, and write $a_k=A_{i_k}$, $u_k=s_{i_k}$ and $v_k=t_{i_k}$.
	Since $t_i\geq\tau s_i$, the first $a_k$ symbols contain at least $\tau a_k/(1+\tau)$ $0$'s and at most $a_k/(1+\tau)$ $1$'s.
	It follows that
	\begin{equation*}
		a_k\leq\frac{1+\tau}{\tau}\varepsilon_kn_k\quad\text{and}\quad
		u_k\geq(1-\varepsilon_k)n_k-\frac{a_k}{1+\tau}\geq\left(1-\frac{1+\tau}{\tau}\varepsilon_k\right)n_k.
	\end{equation*}
	Consequently, we have $u_k\to\infty$ and $a_k/u_k\to0$.
	Let
	\begin{equation*}
		l_k:=\lceil\tau u_k\rceil\leq v_k,\quad N_k:=a_k+u_k+l_k.
	\end{equation*}
	By \eqref{eq-thm-example-1-proof-2}, we have
	\begin{equation*}
		\begin{split}
			\rho\Big(\delta_x^{N_k},\frac{a_k+u_k}{N_k}\mu+\frac{l_k}{N_k}\omega\Big)&\leq\frac{1}{N_k}\bigg(\sum_{j=0}^{a_k-1}d(f^jx,1^{\infty})+\sum_{j=a_k}^{a_k+u_k-1}d(f^jx,1^{\infty})+\sum_{j=a_k+u_k}^{N_k-1}d(f^jx,0^{\infty})\bigg)\\
			&\leq\frac{1}{N_k}\bigg(\sum_{j=0}^{a_k-1}d(f^jx,1^{\infty})+\sum_{j=a_k}^{a_k+u_k-1}d(f^jx,1^{\infty})+\sum_{j=a_k+u_k}^{N_k-1}d(f^jx,0^{\infty})\bigg)\\
			&\leq\frac{2a_k+2+2}{N_k}=\frac{2a_k+4}{N_k}\to0.
		\end{split}
	\end{equation*}
	Since $a_k/u_k\to0$ and $l_k/u_k\to\tau$, we obtain that
	\begin{equation*}
		\delta_x^{N_k}\to\frac{1}{1+\tau}\mu+\frac{\tau}{1+\tau}\omega=\nu.
	\end{equation*}
	Thus $x\in G^{\nu}$, which shows \eqref{eq-thm-example-2-proof-1}.

	By \eqref{eq-thm-example-2-proof-1}, for any $x,y\in G^{\mu}\setminus G^{\nu}$, there exists $j\geq0$ such that $f^jx=f^jy=1^{\infty}$.
	In particular, $d(f^nx,f^ny)=0$ for every $n\geq j$, so $\{x,y\}$ is not a DC1-scrambled pair.
	This shows (3) and completes the proof.
\end{proof}

The following example shows that the DC1-scrambled set obtained in Theorem~\ref{thm-DC1-3} cannot, in general, be chosen inside the corresponding saturated set $G_K$.

\begin{proposition}
	For any $0<\tau<\infty$, there exists a dynamical system $(X,f)$ with non-uniform $\tau$-specification property such that
	\begin{enumerate}
		\item there exist two non-empty compact subsets $K,L\subset\mathcal{M}_f(X)$ and a distal pair $\{p_0,p_1\}\subset X$ satisfying conditions~(1) and~(2) of Theorem~\ref{thm-DC1-3};
		\item $G_K\neq\varnothing$;
		\item $G_K$ has no DC1-scrambled pairs.
	\end{enumerate}
\end{proposition}

\begin{proof}
Let $\mathcal{A}:=\{1,2,3\}$ and consider a one-sided full shift $\{0,1,2,3\}^{\mathbb{N}_0}$ over $\widetilde{\mathcal{A}}:=\{0\}\bigcup\mathcal{A}=\{0,1,2,3\}$.
Let $f$ denote the shift action on $\widetilde{\mathcal{A}}^{\mathbb{N}_0}$.
Let $Y\subset\{1,2\}^{\mathbb{N}_0}$ be the one-sided projection of a non-regular Oxtoby Toeplitz subshift constructed as in \cite[Section~3 and Theorem~3.2]{Williams-1984}.
Then $(Y,f)$ is minimal, whose natural extension is the corresponding two-sided Oxtoby Toeplitz subshift.
Since ergodic measures correspond under the natural extension, $(Y,f)$ has exactly two ergodic measures, denoted by $\nu_1$ and $\nu_2$.
By minimality and ergodicity, $G^{\nu_1}$ and $G^{\nu_2}$ are dense $G_{\delta}$ sets.
Moreover, it follows from the ergodic decomposition theorem that $\mathcal{M}_f(Y)=\mathrm{co}\{\nu_1,\nu_2\}$.
Since $V_f(x)$ is non-empty, compact and connected for every $x\in Y$ by \cite[Proposition~3.8]{DGS1976}, it follows that
\begin{equation*}
    G_{\mathcal{M}_f(Y)}=G^{\mathcal{M}_f(Y)}=G^{\nu_1}\cap G^{\nu_2}
\end{equation*}
is a dense $G_{\delta}$ set.
Finally, the one-sided projection of a Toeplitz point is regularly recurrent.
Hence $Y$ contains no DC1-scrambled pair by \cite[Corollary~4.2 and the paragraph following it, p.~418]{Oprocha-2012}.

Consider forbidden words $\mathcal{F}$ over $\widetilde{\mathcal{A}}$ consisting of
\begin{itemize}
	\item adjacent pairs $\{13,31,23,32\}$;
	\item all finite words over $\{1,2\}$ that do not appear in $Y$;
	\item all words over $\widetilde{\mathcal{A}}$ that are of the form $x_1\cdots x_s0^tx_{s+1}$, where $s,t\in\mathbb{N}$, $x_i\neq 0$ and $t<\tau s$.
\end{itemize}
Define
\begin{equation*}
	X:=X(\mathcal{F}):=\{x\in\widetilde{\mathcal{A}}^{\mathbb{N}_0}:x_ix_{i+1}\cdots x_l\notin\mathcal{F},\;\forall\,0\leq i\leq l\}.
\end{equation*}
Let $Z:=Y\cup\{3^{\infty}\}$.
The above construction coincides with the construction in the proof of Theorem~\ref{example-thm1}, applied to the alphabet $\mathcal{A}$ and the subshift $Z$.
It follows from the proof of Theorem~\ref{example-thm1} (see \cite[\S~4.1]{Lin-Tian-Yu-2025+}) that $(X,f)$ satisfies the non-uniform $\tau$-specification.

Let $\mu=\delta_{3^{\infty}}$ and $K=\mathcal{M}_f(Y)$.
Then $G_K\neq\varnothing$, which shows (2).
Choose a distal pair $\{p_0,p_1\}\subset Y$.
Then
\begin{equation*}
	\overline{\mathrm{co}}(K\cup V_f(p_0)\cup V_f(p_1))=K.
\end{equation*}
Since $\mu$, $\nu_1$ and $\nu_2$ are ergodic, by Proposition~\ref{pre-tv-prop1}, we have
\begin{equation*}
	\rho_{\mathsf{tv}}(\mu,\nu)=1>\frac{\tau}{1+\tau},\quad\forall\nu\in K.
\end{equation*}
In particular, $(X,f)$ satisfies conditions~(1) and~(2) of Theorem~\ref{thm-DC1-3}, with $L=\{\mu\}$.
This shows (1) and (2).

It remains to show that $G_K$ contains no DC1-scrambled pair.
It follows from Theorem~\ref{example-thm1} and Remark~\ref{example-rmk2} that for every $x\in X\setminus\bigcup_{j\geq0}f^{-j}Z$, the set $V_f(x)$ contains a measure not belonging to $\mathcal{M}_f(Z)$.
Since $K\subset\mathcal{M}_f(Y)\subset\mathcal{M}_f(Z)$, we have
\begin{equation*}
	G_K\subset\bigcup_{j\geq0}f^{-j}Z.
\end{equation*}
For any $x\in G_K$, if $f^jx=3^{\infty}$ for some $j\geq0$, then $V_f(x)=\{\mu\}\neq K$.
Consequently,
\begin{equation*}
	G_K\subset\bigcup_{j\geq0}f^{-j}Y.
\end{equation*}
Moreover, $G_K$ has no DC1-scrambled pairs.
Indeed, suppose for a contradiction that $\{x,y\}\subset G_K$ is a DC1-scrambled pair.
There exists $j\geq0$ such that $f^jx,f^jy\in Y$.
This implies that $\{f^jx,f^jy\}$ is a DC1-scrambled pair, which contradicts the fact that $Y$ has no DC1 pairs.
This completes the proof.
\end{proof}


\begin{thebibliography}{99}


\bibitem{Berthe-Holton-Zamboni-2006}
V. Berth\'e, C. Holton and L. Q. Zamboni,
{\it Initial powers of Sturmian sequences},
Acta Arith. {\bf 122} (2006), no.~4, 315--347.

\bibitem{Bowen-1971}
R. Bowen,
{\it Periodic points and measures for Axiom {$A$} diffeomorphisms},
Trans. Amer. Math. Soc. {\bf 154} (1971), 377--397.


%\bibitem{Burguet-2020}
%D. Burguet,
%{\it Topological and almost Borel universality for systems with the weak specification property},
%Ergodic Theory Dynam. Systems. 40 (8) (2020), 2098-2115.

\bibitem{Chen-Tian-2021}
A. Chen and X. Tian,
{\it Distributional chaos in multifractal analysis, recurrence and transitivity},
Ergodic Theory Dynam. Systems {\bf 41} (2021), 349--378.

\bibitem{Chen-Kupper-Shu-2005}
E. Chen, T. K\"{u}pper and L. Shu,
{\it Topological entropy for divergence points},
Ergodic Theory Dynam. Systems {\bf 25} (2005), no.~4, 1173--1208.

%\bibitem{Dateyama-1983}
%M. Dateyama,
%{\it Invariant measures for homeomorphisms with almost weak specification},
%Probability theory and mathematical statistics (Tbilisi, 1982), Lecture Notes in Math. vol. 1021, Springer, Berlin, 1983, 93-96.

\bibitem{DGS1976}
M. Denker, C. Grillenberger and K. Sigmund,
{\it Ergodic Theory on Compact Spaces},
Lecture Notes in Mathematics, vol.~527, Springer-Verlag, 1976.

%\bibitem{Grillenberger-1993}
%C. Grillenberger,
%{\it Constructions of strictly ergodic systems},
%Z. Wahrscheinlichkeitstheorie verw Gebiete 25, 323-334 (1973).

\bibitem{Hou-Tian-Zhao-2025}
X. Hou, X. Tian and X. Zhao,
{\it Chaoticity of generic points for ergodic measures in hyperbolic systems and beyond},
J. Differential Equations {\bf 436} (2025), Paper No. 113236, 48 pp.

\bibitem{Hou-Tian-Zhao-2026}
X. Hou, X. Tian and X. Zhao,
{\it Distributional chaos in shrinking target problems},
Ergodic Theory Dynam. Systems {\bf 46} (2026), no.~4, 1020--1042.

\bibitem{Kwietniak-Lacka-Oprocha-2016}
D. Kwietniak, M. \L\c{a}cka and P. Oprocha,
{\it A panorama of specification-like properties and their consequences},
in {\it Dynamics and Numbers}, Contemp. Math. {\bf 669}, Amer. Math. Soc., Providence, RI, 2016, 155--186.

\bibitem{Li-Wu-2014}
J. Li and M. Wu,
{\it Generic property of irregular sets in systems satisfying the specification property},
Discrete Contin. Dyn. Syst. {\bf 34} (2014), no.~2, 635--645.

\bibitem{Li-Yorke-1975}
T. Y. Li and J. A. Yorke,
{\it Period three implies chaos},
Amer. Math. Monthly {\bf 82} (1975), no.~10, 985--992.

\bibitem{Lin-Tian-Yu-2024}
W. Lin, X. Tian and C. Yu,
{\it Similarities and differences between specification and non-uniform specification},
Ergodic Theory Dynam. Systems {\bf 44} (2024), no.~12, 3501--3529.

\bibitem{Lin-Tian-Yu-2025+}
W. Lin, X. Tian and C. Yu,
{\it Variations of topological theory and ergodic theory via gap function in non-uniform specification},
arXiv:2508.17800v2, 2025.

%\bibitem{Marcus-1980}
%B. Marcus,
%{\it A note on periodic points for ergodic toral automorphisms},
%Monatsh. Math. 89 (2) (1980), 121-129.

\bibitem{Oprocha-2009}
P. Oprocha,
{\it Distributional chaos revisited},
Trans. Amer. Math. Soc. {\bf 361} (2009), no.~9, 4901--4925.

\bibitem{Oprocha-2012}
P.~Oprocha,
{\it Minimal systems and distributionally scrambled sets},
Bull. Soc. Math. France {\bf 140} (2012), no.~3, 401--439.

\bibitem{Pavlov-2016}
R. Pavlov,
{\it On intrinsic ergodicity and weakenings of the specification property},
Adv. Math. {\bf 295} (2016), 250--270.

%\bibitem{Pavlov-2019}
%R. Pavlov,
%{\it On controlled specification and uniqueness of the equilibrium state in expansive systems},
%Nonlinearity. 32 (7) (2019), 2441-2466.

%\bibitem{Quas-Soo-2016}
%A. Quas and T. Soo,
%{\it Ergodic universality of some topological dynamical systems},
%Trans. Amer. Math. Soc. 368 (6) (2016), 4137-4170.

\bibitem{Schweizer-Smital-1994}
B. Schweizer and J. Sm\'ital,
{\it Measures of chaos and a spectral decomposition of dynamical systems on the interval},
Trans. Amer. Math. Soc. {\bf 344} (1994), no.~2, 737--754.

\bibitem{Thompson-2010}
D. J. Thompson,
{\it The irregular set for maps with the specification property has full topological pressure},
Dyn. Syst. {\bf 25} (2010), no.~1, 25--51.

\bibitem{Thompson-2012}
D. J. Thompson,
{\it Irregular sets, the {$\beta$}-transformation and the almost specification property},
Trans. Amer. Math. Soc. {\bf 364} (2012), no.~10, 5395--5414.

\bibitem{Villani-2009}
C. Villani,
{\it Optimal transport. Old and new},
Grundlehren der mathematischen Wissenschaften, vol.~338, Springer-Verlag, 2009.

\bibitem{Williams-1984}
S.~Williams,
{\it Toeplitz minimal flows which are not uniquely ergodic},
Z. Wahrscheinlichkeitstheorie verw. Gebiete {\bf 67} (1984), 95--107.


\end{thebibliography}
\end{document}